\documentclass[a4paper,12pt]{amsart} 
\usepackage{t1enc,amsthm,amssymb,color,tikz,graphics,epsfig,multicol,mathrsfs}
\usepackage[latin1]{inputenc}
\usepackage[english]{babel}
\usepackage[colorlinks=true, breaklinks=true, linkcolor=red, citecolor=blue, urlcolor=blue]{hyperref}

\numberwithin{equation}{section}
\newtheorem{theorem}{Theorem}[section]
\newtheorem{lemma}[theorem]{Lemma}
\newtheorem{corollary}[theorem]{Corollary}
\newtheorem{proposition}[theorem]{Proposition}
\newtheorem{question}[theorem]{Question}
\newtheorem{definition}[theorem]{Definition}
\newtheorem{remark}[theorem]{Remark}
\newtheorem{formula}[theorem]{Formula}
\newtheorem{exercise}[theorem]{Exercise}

\newcommand{\forestftwo}{
\begin{tikzpicture}[baseline=.4cm]
\draw (0,0)--(0,1);
\draw (0,2/3)--(1/3,1);
\draw (-1/3,0)--(-1/3,1);
\draw (1/3,0)--(2/3,1);
\draw (2/3,0)--(1,1);
\end{tikzpicture}
}
\newcommand{\xthirteen}{
\begin{tikzpicture}[baseline=.8cm]
\draw (0,0)--(0,-.5);
\draw (0,0)--(-2,2);
\draw (0,0)--(2,2);
\draw (-1.5,1.5)--(-1,2);
\draw (1,1)--(0,2);
\draw (.5,1.5)--(1,2);
\node at (1.2,.9) {$7$};
\node at (.3,1.4 ) {$3$};
\node at (-1.7,1.4) {$6$};
\node at (-.3,-.1) {$13$};
\end{tikzpicture}
}

\def\hpic #1 #2 {\mbox{$\begin{array}[c]{l} \epsfig{file=#1,height=#2}
\end{array}$}}
 
\def\vpic #1 #2 {\mbox{$\begin{array}[c]{l} \epsfig{file=#1,width=#2}
\end{array}$}}

\newcommand {\5}{\vskip 5pt}

\newcommand{\act}{\curvearrowright}
\newcommand{\cA}{\mathcal A}
\newcommand{\cAF}{\mathcal{AF}}
\DeclareMathOperator{\Ad}{Ad}
\newcommand{\bu}{\bullet}
\newcommand{\cB}{\mathcal B}
\newcommand{\C}{\mathbf C}
\newcommand{\cC}{\mathcal C}
\newcommand{\fC}{\mathfrak C}
\DeclareMathOperator{\Diff}{Diff}
\newcommand{\cF}{\mathcal F}
\newcommand{\fH}{\mathfrak H}
\newcommand{\scrH}{\mathscr H}
\DeclareMathOperator{\Hilb}{Hilb}
\DeclareMathOperator{\id}{id}
\newcommand{\cI}{\mathcal I}
\newcommand{\cJ}{\mathcal J}
\newcommand{\la}{\lambda}
\DeclareMathOperator{\Mor}{Mor}
\DeclareMathOperator{\Leb}{Leb}
\newcommand{\N}{\mathbf N}

\DeclareMathOperator{\Rot}{Rot}
\newcommand{\R}{\mathbf R}
\newcommand{\cR}{\mathcal R}
\newcommand{\bfS}{\mathbf S}
\newcommand{\cSF}{\mathcal{SF}}
\newcommand{\fT}{\mathfrak T}
\DeclareMathOperator{\target}{target}
\newcommand{\un}{\underline}
\newcommand{\varep}{\varepsilon}
\newcommand{\Z}{\mathbf Z}

\begin{document}
\title{Pythagorean representations of Thompson's groups.}
\author{Arnaud Brothier and Vaughan F. R. Jones}
\address{Arnaud Brothier\\ School of Mathematics and Statistics, University of New South Wales, Sydney NSW 2052, Australia}
\email{arnaud.brothier@gmail.com\endgraf
\url{https://sites.google.com/site/arnaudbrothier/}}
\address{Vaughan F. R. Jones\\ Vanderbilt University, Department of Mathematics, 1326 Stevenson Center Nashville, TN, 37240, USA}\email{vaughan.f.jones@vanderbilt.edu}
\thanks{A.B. was partially supported by the European Research Council Advanced Grant 669240 QUEST. 
V.J. is not supported by an NSF grant.}

\begin{abstract}
We introduce the Pythagorean C*-algebras and use the  category/functor method to construct  unitary representations of Thompson's groups from representations of them. We calculate several examples.
\end{abstract}
\maketitle

\section{Introduction}
Let $F$ and $T$ be the Thompson's groups as usual-see \cite{Cannon-Floyd-Parry96}, and $F_n$ and $T_n$ their $n$-ary versions for $n\geq 2$. (So that $F=F_2$, $T=T_2$.)
In \cite{Jones17-Thompson} an action of $F_n$ arose from a \emph{functor} from the category $\mathcal F_n$ of $n-ary$ planar forests, whose objects are natural numbers and whose morphisms are planar forests, to another category $\cC$.
Forests decorated with cyclic permutations of their leaves give categories $\cAF_n$ of affine $n$-ary planar forests for which functors from $\cAF_n$ give actions of $T_n$.

The representations studied in \cite{Jones16-Thompson} came from functors $\Phi$ to a tensor category $\cC$ with $$\Phi(n)=\otimes^n V$$ for some object $V\in \mathcal C$ and an element $R\in \Mor(V,V\otimes V)$ which generates the action of forests by letting
$$\Phi(f)=(\otimes^{i-1}\id)\otimes R\otimes (\otimes^{n-i}\id)$$
where $f$ is the binary planar forest with $n$ roots and $n+1$ leaves
and a tree connecting the $i$th root to two leaves.

A simpler situation (in some sense the ``classical'' version where the one just described is the ``quantum'' version) is where $\cC=\Hilb$ is the category of Hilbert spaces (with isometries as morphisms) and
$$\Phi(n)=\oplus^n \mathfrak H, \ \Phi(f)=(\oplus^{i-1}\id)\oplus R\oplus (\oplus^{n-i}\id),$$ 
where $\fH\in\Hilb, \ R\in\Mor(\fH,\fH\oplus\fH),$ and $f$ is as above.
Then $R$ is necessarily of the form $A\oplus B$ with 
$$|A|^2+|B|^2=\id, \text{ where } | A |:=\sqrt{A^*A}.$$
For the Thompson's groups $F_n$ and $T_n$ we would use an $R$ of the form  $\oplus_{i=1}^n A_i$ where $\sum_{i=1}^n |A_i|^2=\id$ so we call the representations on the direct limit Hilbert space $\scrH$ the \emph{Pythagorean} representations of the Thompson's groups and we define the corresponding universal C$^*$-algebra.

\begin{definition}\label{def:P_n}
Let $P_n$ be the universal $C^*$-algebra generated $A_1,A_2,\cdots, A_n$ subject to the relation $$A_1^*A_1+A_2^*A_2+\cdots + A_n^*A_n=\id.$$
\end{definition}

Note that there is indeed a universal C$^*$-norm defining $P_n$ since in any *-representation on Hilbert space the norms of all the $A_i$ are at most $1$. 
The C$^*$-algebra $P_n$  has many quotients including continuous functions on the projective space $\C P^{n-1}$, the Cuntz algebra $O_n$ and the full free group algebra for the free group on $n$ generators.

Cyclic groups act on direct sums of a fixed Hilbert space in a way compatible with $\Phi$ just defined so one also obtains unitary representations of $T_n$ from the same data.

In this paper, we will investigate these representations of $F$ and $T$ for some choices of $A$ and $B$.

\subsection*{Acknowledgement}
We are grateful to Anna Marie Bohmann, Georges Skandalis and Ruy Exel for valuable comments and discussions.
We thank the New Zealand Mathematics Research Institute for its generous support.

\section{Definitions}\label{sec:def}

A binary planar forest is the isotopy class of a disjoint union of binary trees embedded
in $\R^2$ all of whose roots lie on $(\R,0)$ and all of whose leaves lie on $(\R,1)$. The isotopies are supported
in the strip $(\R,[0,1])$. Binary planar forests form a category
in the obvious way with objects being $\N$ whose elements are identified with isotopy classes of sets of points on a line and whose morphisms are the forests which can be composed by stacking a forest in 
$(\R,[0,1])$ on top of another, lining up the leaves of the one on the bottom with 
the roots of the other by isotopy then rescaling the $y$  axis to return
to a forest in  $(\R,[0,1])$. The structure is of course actually combinatorial but it is very useful to think of it in the way we have described.

We will call this category $\cF$.

\begin{definition}\label{genforest} Fix $n\in \N$. For each $i=1,2,\cdots,n$ let $f_{i,n}$ (or simply $f_i$ if the context is clear) be the planar
binary forest with $n$ roots and $n+1$ leaves consisting of straight
lines joining $(k,0)$ to $(k,1)$ for $1\leq k\leq i-1$ and $(k,0)$ to 
$(k+1,1)$ for $i+1\leq k\leq n$, and a single binary tree with root
$(i,0)$ and leaves $(i,1)$ and $(i+1,1)$ thus:
$$f_{2,4}=\forestftwo.$$
\end{definition}

Note that any element of $\cF$ is in an essentially unique way a 
composition of morphisms $f_i$, the only relation being $\Phi(f_j)\Phi(f_i)=\Phi(f_i)\Phi(f_{j-1}) \mbox{   for  } i<j-1$.
Call $e$ the unique morphism in $\cF$ from $1$ to $1$. 
The set of morphisms from $1$ to any number $n$ in $\cF$ is the set of binary planar rooted trees $\fT$ and is a \underline{directed set} with $s\leq t$ if and only if there is $f\in \cF$ with $t=f\circ s$.

Given a functor $\Phi:\cF\rightarrow \cC$ where the objects of $\cC$ are sets, we define the direct system $S_\Phi$ which associates to each $t  \in \fT$ with $n$ leaves the set $\Phi(n)$. 
For each $s\leq t$ we need to give a morphism $\iota_s^t$. 
For this observe that there is a unique $f\in \cF$ for which $t=f\circ s$ so we may
define $\iota_s^t$ to be $\Phi(f)$. As in \cite{Jones16-Thompson} we consider the direct limit:
$$ \underset{\rightarrow} \lim S_\Phi=\{(t,x)| t\in  \fT, x\in \Phi(\target(t))\} / \sim$$ 
where $(t,x)\sim (s,y)$ if and only if there are $f,g\in\cF$ with $f\circ t=g \circ s$ and $\Phi(f)(x)=\Phi(g)(y)$.
Denote by $\frac{t}{x}$ 
% or $(t,x)$
 the equivalence class of $(t,x)$ inside this quotient. If the $\iota_s^t$ are all injections, as they will be in our 
 Pythagorean case, we may identify each $\Phi(n)$ with its image in the direct limit so that $(t,x)$ may 
 also be used to represent $\frac{t}{x}$.

The limit $ \underset{\rightarrow} \lim S_\Phi$ will inherit some structure from the category $\cC$. Our main interest here
 is in the category $\cC=\Hilb$ of Hilbert spaces with isometries for morphisms.
The direct limit $ \underset{\rightarrow} \lim S_\Phi$ will be a pre-Hilbert space which may be completed to a Hilbert space which we will also call the direct limit unless special care is required.
 
Note that this is a slight modification of the definition of \cite{Jones16-Thompson} where
$S_\Phi(t)$ was $\Mor(\Phi(e),\Phi(t))$. 
This was necessary in \cite{Jones16-Thompson} to make $\Phi(t)$ a \emph{set} since we were dealing with abstract tensor categories. 
Sometimes $\Mor(\Phi(e),\Phi(t)) $  and $\Phi(\target(t))$ can be naturally identified in which case the definition is the same as in \cite{Jones16-Thompson}. 
Such is the case for the identity functor from $\cF$ to itself or when the target category is the rectangular category of an irreducible planar algebra.

As was observed in \cite{Jones16-Thompson}, $\cF$ has the required properties so that $ \underset{\rightarrow} \lim S_\Phi$ is, when $\Phi$ is the "identity" functor (taking an object $n$ of $\cF$ to $\Mor(1,n)$ and morphisms to composition in the obvious way), the Thompson's group $F$, which is thus the \emph{group of fractions} of $\cF$, see \cite{Cannon-Floyd-Parry96} or in a  language closer to ours,\cite[Section 7.2]{Belk04}.
The element $a\in \Mor(1,\target(b))$ is written $\displaystyle \frac{b}{a}$.

Moreover, for any other functor $\Phi$, $ \underset{\rightarrow} \lim S_\Phi$ carries a natural action of $F$ defined as follows:
$$\frac{t}{s}(\frac{s}{x})=\frac{t}{x}$$ where $s,t\in \fT$ with $\target(s)=\target(t)=n$ and $x\in \Phi(n)$.  
A Thompson group element given as a pair of trees with $m$ leaves and an element of $ \underset{\rightarrow} \lim S_\Phi$ given as a pair (tree with $n$ leaves, element of $\Phi(n)$) may appear to not be composable by the above formula, but they can always be ``stabilised'' to be so within their equivalence classes. 
 
The Thompson group action preserves the structure of $ \underset{\rightarrow} \lim S_\Phi$ so for instance in the Hilbert space case the representations are unitary.

For the reader who wants to follow this paper in detail, we propose the following exercise. 
Let $\cC$ be the category of sets, $\sigma$ be a set with one element $x$ and $\Phi$ be the functor from $\cF$ to $\cC$ defined by:
\begin{enumerate}
\item $\Phi(n)=\coprod_1^n \sigma$=$\{x_1,x_2,\cdots,x_n\}$.\\
\item For each $n$ and $i=1,2,\cdots n$, 
$\Phi(f_i)(x_j)= \begin{cases} x_j &\mbox{  if } j< i \\ 
x_{j+1} & \mbox{  if } j\geq i\end{cases}.$
\end{enumerate}
(recall that the $f_i$ are the generators of $\mathcal F$ as in definition \ref{genforest}.)

\begin{exercise} With $\Phi$ as above, calculate the stabiliser of $(t,x_i)\in 
 \underset{\rightarrow} \lim S_\Phi$ for a tree $t$ with $n$ 
 leaves and an element $x_i$ of $\coprod_1^n \sigma$. (Note that an element of 
 $\Mor(\Phi(1),\Phi(n))$ is the same thing as an element of
$ \coprod_1^n \sigma$.)
\end{exercise}

We now define the Pythagorean representations.

\begin{definition}\label{pythagoras} Given a Hilbert space $\fH$ and a pair of bounded operators $(A,B)$ on $\fH$ satisfying the Pythagorean equation $A^*A+B^*B=\id$ we define the functor $\Phi=\Phi_{A,B}:\cF\to\Hilb$ as follows:
\begin{enumerate}
\item $\Phi(n)=\oplus_1^n \fH.$\\
\item For each $n$ and $i=1,2,\cdots n$, \\
$\Phi(f_{i,n})(\oplus_{j=1}^n \xi_j)= \oplus_{j=1}^{n+1} \eta_j\mbox{  with  } 
\eta_j=\begin{cases} \xi_j &\mbox{  if } j< i \\ 
A(\xi_{i}) & \mbox{  if } j= i\\
B({\xi_{i}}) & \mbox{ if } j=i+1\\
\xi_{j-1} &\mbox{ if }j>{i+1}\end{cases}.$
\end{enumerate}
\end{definition}

\begin{proposition} 
With notation as in definition \ref{pythagoras},
$$\Phi(f_j)\Phi(f_i)=\Phi(f_i)\Phi(f_{j-1}) \mbox{   for  } i<j-1$$
and the $\Phi(f_i)$ are isometries so $\Phi$ extends to a functor
from $\mathcal F$ to $\Hilb$.
\end{proposition}

\begin{definition} 
With $\fH, A,B$ and $\Phi$ as above we call the unitary representation $\pi=\pi_{A,B}$ of $F$ on the direct limit $\scrH$ the \emph{Pythagorean representation} given by $(A,B)$.
Elements of $\scrH$ are written $\frac{t}{\xi}$. 
%or $(t,\xi)$
 %and the one of $F$ by $\frac{t}{s}$ with $s,t\in\fT$.
We identify $\fH:=\Phi(1)$ and $\fH_t:=\{\frac{t}{\eta}:\eta\in\Phi(\target(t))\}$ as subspaces of $\scrH$.
 
More generally if $A_i$ satisfy $\sum_{i=1}^n A_i^*A_i=\id$ we call the representation $\pi_{A_1,\cdots, A_n}$  of $F_n$ (using the $A_i$ to represent the category of planar $n-ary$ forests) the \emph{Pythagorean representation} given by $A_1,A_2,\cdots A_n$.
\end{definition}

We will sometime write $p\bu q$ as the forest obtained by concatenating horizontally a forest $p$ to the left of a forest $q$.
For example, if $p$ is the forest with $m$ straight lines, then $p\bu f_{i,n}=f_{i+m,n+m}$ and $f_{i,n}\bu p = f_{i,n+m}$.

\section{The coefficients}\label{sec:coefficients}
The direct limit may be quite tricky to determine explicitly. 
A useful point of access will be the \emph{coefficients} of the representation $(\pi,\scrH)$, i.e.~
functions on the group of the from $\langle \pi(g)\xi,\eta\rangle$ where $\xi,\eta\in\scrH$.

\begin{definition} 
Suppose we are given $A$ and $B$ acting on $\fH$ with $A^*A+B^*B=\id$. 
Choose a unit vector $\Omega\in \fH$ and call its image in the direct limit $\scrH$ the \emph{vacuum}, also denoted $\Omega$. (Written in full it would be the class in the direct limit of $(e,\Phi(\id)(\Omega))$.)
\end{definition}

Let us calculate the coefficient $\langle\pi(g)\Omega, \Omega\rangle$ (which determines the Pythagorean representation of $F$ on the linear span of the $\pi(h)\Omega$ as $h$ varies in $F$).

 To this end observe first that each leaf $\ell$ of a binary planar rooted tree $t$ with $n$ leaves is indexed by a sequence $(e_1,e_2,\cdots, e_k)$ of $0$'s and $1$'s according to whether the branch to the leaf turns left ($0$) or right ($1$) at the $i$th vertex from the root.
For $1\leq i\leq k$  let  
$$X_i^{\ell}=\begin{cases} A&\mbox{  if } e_i=0\\ 
B & \mbox{  if } e_i=1 \end{cases}$$

\begin{definition} \label{leafoperator}
With notation as above set $$\mathcal A_{\ell}^t=X_k^{\ell}X_{k-1}^{\ell}\cdots X_1^{\ell}$$
We will freely identify a tree $t$ with its set of leaves $\{\ell\}$ with a standard dyadic partition $\cI=\cI_t$ with intervals $\{I\}.$
Note that the operator $\cA^t_\ell$ only depends on the corresponding interval $I$ and not on the partition $\cI_t$ as there is a unique path to any standard dyadic interval inside the infinite full binary planar tree, see \cite{Cannon-Floyd-Parry96} for details.
We will often write $\cA_I$ instead of $\cA^t_\ell$.
\end{definition}
\5
By definition, $\Omega$ is the (class of the) pair $(t,\underset{\ell}\oplus v^t_{\ell})$ with $v^t_{\ell}=\cA_{\ell}^t\Omega$.
Thus, if the pair of trees $\displaystyle \frac{t}{s}$ is an element of $F$, then each leaf $\ell$ of $t$ is identified with a leaf of $s$, which we also call $\ell$, and  $(t,\underset{\ell}\oplus v^s_{\ell})$ represents $\pi(g)(\Omega)$. 
So $$\langle \pi(g)\Omega,\Omega\rangle= \sum_{ \ell \mbox{  a leaf of  }t }\langle \cA_\ell^s \Omega,\cA_\ell^t\Omega\rangle.$$
 
 \begin{definition}\label{pathoperator}
 For $\displaystyle \frac{t}{s}$ and $\ell$ as above set $$\cB^{t,s}_\ell=(\cA_\ell^t)^*\cA_\ell^s \text{ or simply } \cB_\ell=(\cA_\ell^t)^*\cA_\ell^s.$$
 \end{definition}
 
 We obtain the following:
 \begin{formula}\label{formula} $$\langle \pi(g)\Omega,\Omega\rangle= \sum_{\ell \mbox{ a leaf of } t}\langle\cB^{t,s}_\ell \Omega,\Omega\rangle.$$
 \end{formula}
 
 The following example should make this formula clear: \\
 Let $g\in F$ be given by the pair of trees below:
 
 \mbox{   \qquad} $\displaystyle \frac{t}{s}=  $\hspace {0.1in}  \vpic {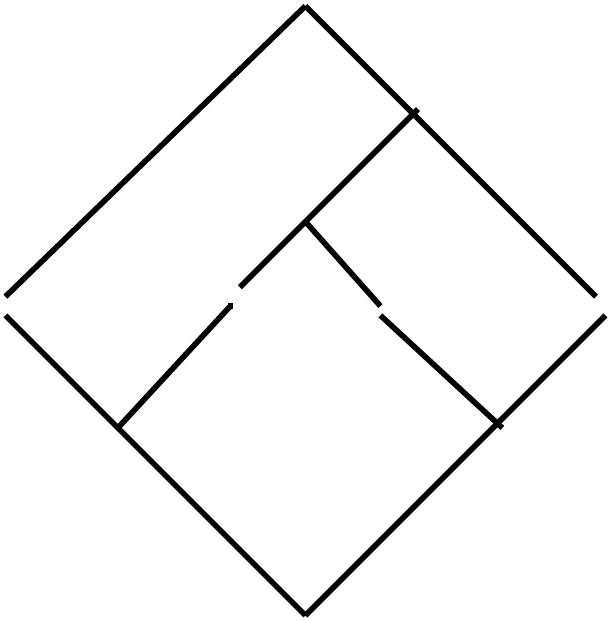} {1in}

 Then $$\langle \pi(g)\Omega,\Omega\rangle= \langle 
  (A^*AA+B^*A^*A^*BA+B^*A^*B^*AB+B^*B^*BB)\Omega , 
  \Omega\rangle $$
  
  Verbally one could express the formula as follows:\\
  `` Arrange the pair of trees one on top of the other. Label the edges of the bottom tree by $A$ or $B$ according to whether they are left or right edges, and similarly the top tree by $A^*$ or $B^*$. For each path on the pair of trees from top to bottom form the operator given by the product of the operators on each edge, take its $\Omega-\Omega$ coefficient  and sum.''

\section{Thompson's group $T$}
\subsection{A Pythagorean representation extends to larger groups}\mbox{  }

We briefly recall how to get $T$ in the category picture. 
See \cite{Jones16-Thompson} or \cite{Graham-Lehrer98-TLJ} for more details.
The objects of the category of affine binary planar forests $\cAF$ are the same as those of $\cF$ i.e.~sets of $n$ points on a line up to isotopy, identified with $\N$. 
A morphism of $\cAF$ is a pair $(f,k)$ where $f$ is a binary planar forest with $m$ roots and $n$ leaves, on two parallel lines, up to isotopy, and $k$ is an element of $\Z/n\Z$ represented by an integer between $0$ and $n-1$.
Morphisms $(f,k)\in\cAF(m,n)$ and $(g,\ell)\in\cAF(n,p)$ are composed as follows.
First form the planar rooted forest $g\circ_k f$ with $m$ roots and $p$ leaves by attaching the leaves of $f$ to the roots of $g$ in cyclic order starting by attaching the $(k+1)$th leaf of $f$ to the first root of $g$. 
Roots and leaves are counted from left to right starting from $1$ on the left. 
Thus the $(n+1-k)$th root of $g$ is attached to the first leaf of $f$ and so on.  
The composition of the two morphisms is then $(g\circ_k f,(k'+\ell) \mod p)$ where $k'$ is equal to  the sum of the number of leaves of the $k$ last trees of $g$.

We see that $\cAF(n,n)$ is isomorphic to the cyclic group $\Z/n\Z$ and that the category $\cAF$ satisfies the conditions to have a group of fractions which will in this case be pairs of trees with the same number of vertices, each with a distinguished leaf, up to cancelling 
carets and a common cyclic group action. Stabilising by a morphism from $n$ to $n$, we see that the distinguished leaf of the first tree may be supposed to be the leftmost. 
We  thus obtain the usual "pair of trees" picture of $T$ as in \cite{Cannon-Floyd-Parry96}.

\begin{remark} 
If we consider $T$ as acting on the circle $\bfS$ and that $t=s$ are equal to the full binary tree with $2^n$ leaves, then the fraction obtained from the pair $((t,k) , (t,0))$ acts as the rotation of angle $k 2^{-n}.$
\end{remark}

A Pythagorean pair $(A,B)$ on $\fH$ gives a representation of $T$ in the same way as we did for $F$. 
Just use $(A,B)$ to go between direct sums of $\fH$ according to the forest and distinguished leaf.

It is important to note that the Hilbert space $\scrH$ obtained from $\cAF$ is the \emph{same} one as that obtained from $\cF$ since all the  vector spaces corresponding to a fixed tree are identified via rotations, and the category $\cF$ is contained in $\cAF$ by making the leftmost leaf of 
a forest the distinguished one. 
It is also clear that the restriction of the representation of $T$ to $F<T$ is precisely the representation we have constructed for $F$.

Thus we will use the same notations $\pi$ and $\rho$ (below) for the representations of $T$ and $F$.

One may calculate the coefficients for Thompson's group $T$ in the same way and they are given by the same formula. 
The only difference is that the identification of the leaves of the pairs of trees may be dislocated so that the leftmost leaf of the bottom tree may be identified with any leaf of the top tree. 
For instance the element  $g\in T$ which rotates by $\pi/2$  is specified by the following:
 \vskip 5pt
 \mbox{   \qquad}\hspace {1in}  \vpic {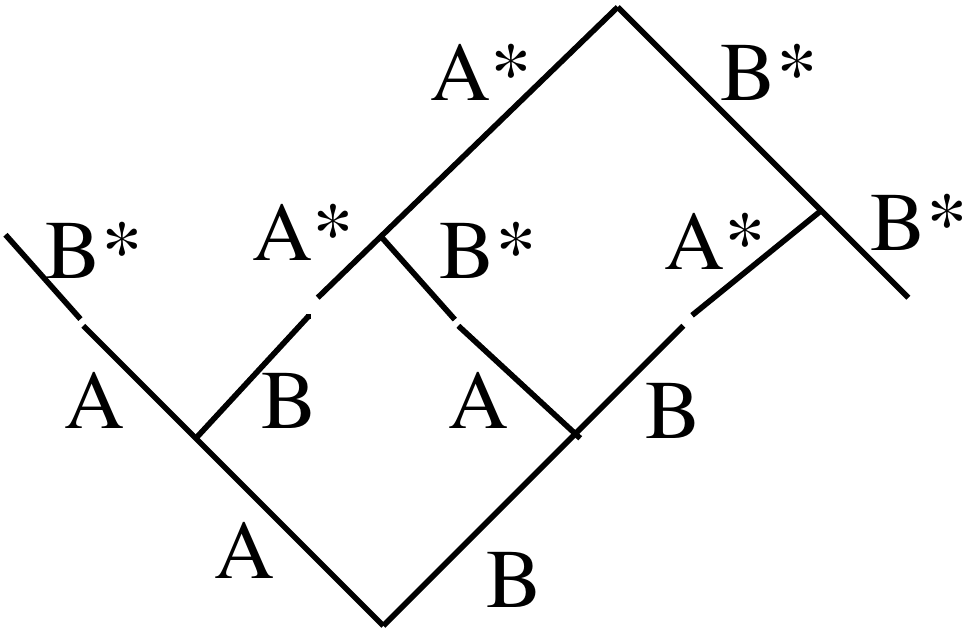} {1.8in}

(Note that the rightmost and leftmost edges are identified.)
We see that 
$$\langle \pi(g)\Omega,\Omega\rangle= \langle  (A^*A^*BA + A^*B^*AB+B^*A^*BB+B^*B^*AA) \Omega , \Omega \rangle. $$

\begin{remark}\label{rem:V}
Thompson's group $V$ can be defined in a similar way in the category picture.
Let $\cSF$ be the category of symmetric forests with set of objects $\N$ and morphisms $\cSF(n,m)=\cF(n,m)\times S_m$ that is forests times the symmetric group of $m$ elements, $n,m\geq 1$.
Composition of morphisms is performed just like it was  for $\cAF$ and the resulting group of fractions is isomorphic to Thompson's group $V$.
If $\pi$ is a Pythagorean representation of $F$, then we can extend it to a representation of $V$ with the following formula:
$$\pi \left( \frac{ ( t , \tau ) }{ ( s , \sigma ) } \right) ( \frac{ s }{ \oplus_\ell \xi_\ell } ) = \frac{ s }{ \oplus_\ell \xi_{ \sigma^{-1}\tau( \ell ) } } .$$ 
One can define similarly $V_n$ and extend Pythagorean representations in the $n$-ary version.
We will restrict our study to representations of $F$ and $T$ but it is remarkable that any Pythagorean representation extends to $V$ and further interesting questions can be raised in this context.
\end{remark}

\subsection{The rotation}\label{sec:rotation}\mbox{ }

As we mentioned before if $t_n$ is the full binary tree with $2^n$ leaves, then Thompson's group $T$ contains the rotation $r_n=\frac{ ( t_n , 1) }{ ( t_n , 0 )}$ by an angle $2^{-n}$ when it acts on the circle $\bfS\simeq \R/\Z$.
Hence, if $\xi = \displaystyle \frac{ t_n} { (\xi_1 , \cdots , \xi_{ 2^n }) } \in \scrH$, then $\pi(r_n)\xi =\displaystyle \frac{ t_n}{ (\xi_2 , \xi_3,\cdots , \xi_{2^n} , \xi_1)}$.
We are interested in knowing if $\pi(r_n)$ has a weak operator limit when $n$ tends to infinity.
We start by observing that the first component of the inner product $\langle\pi(r_n)\xi,\eta\rangle$ for $\xi,\eta\in \fH_{t_n}$ tends to zero at infinity.

\begin{lemma}\label{lem:weak-limit}
The sequence of operators $(B^n)^*A^n$ tends to zero for the weak operator topology of $B(\fH)$.
\end{lemma}
\begin{proof}
Consider a Pythagorean couple $A,B$ and define the spectral projections $p_A:=\chi_{[2/3,1]}(A^*A)$ and $p_B:=\chi_{[2/3,1]}(B^*B)$ associated to the interval $[2/3,1]$.
Note that $A^*A$ and $B^*B$ are positive operators smaller than the identity.
Therefore, the projection $p_A^\perp:=\id-p_A$ is the spectral projection of $A^*A$ associated to the interval $[0,2/3).$

We claim that the sequence $(p_A^\perp A^n)_n$ tends to zero for the strong operator topology.
Consider a vector $\zeta\in \fH$ and observe that 
\begin{align}\nonumber
\Vert A\zeta\Vert^2 & = \langle A\zeta,A\zeta\rangle = \langle A^*A\zeta,\zeta\rangle = \langle p_A A^*A\zeta,\zeta\rangle + \langle p_A^\perp A^*A\zeta,\zeta\rangle\\ \nonumber
& = \langle  A^*A p_A\zeta,p_A\zeta\rangle + \langle  A^*Ap_A^\perp\zeta,p_A^\perp\zeta\rangle \text{ since $A^*A$ and $p_A$ commute} \\ \label{equa:PROJ}
& \leq \Vert p_A\zeta\Vert^2 + \frac{2}{3} \Vert p_A^\perp\zeta\Vert^2 
 = \Vert \zeta\Vert^2 - \frac{1}{3}\Vert p_A^\perp\zeta\Vert^2. 
\end{align}
Fix a vector $\xi\in \fH$ and observe that $(\Vert A^n\xi\Vert)_n$ is a decreasing sequence since $\Vert A\Vert\leq 1.$
Suppose that $\Vert p_A^\perp A^n\xi\Vert$ does not tend to zero.
Then there exists $C>0$ and a strictly increasing sequence $(n_k)_k$ such that $\Vert p_A^\perp A^{n_k}\xi\Vert\geq \sqrt C$ for any $k$.
We obtain that 
$$\Vert A^{n_{k+1}}\xi\Vert^2  \leq \Vert A^{n_k + 1} \xi\Vert^2 = \Vert A(A^{n_k}\xi)\Vert^2  \leq \Vert A^{n_k}\xi \Vert^2 - C/3 \text{ by } \eqref{equa:PROJ}. $$
By iterating the process we get $\Vert A^{n_{k+1}}\xi\Vert^2\leq 1-kC/3$, a contradiction since eventually $1-kC/3$ is negative.
This proves the claim.

The Pythagorean equation implies that $p_A$ and $p_B$ commute and $p_Ap_B=0$.
Choose some unit vectors $\xi,\eta\in \fH$ and a real number $\varepsilon>0.$ 
By the claim there exists $N$ satisfying
$$ \Vert A^n \xi - p_A A^n \xi \Vert +  \Vert B^n \eta - p_B B^n \eta \Vert <\varepsilon, \forall n\geq N.$$
We obtain for all $n\geq N$:
\begin{align*}
\vert \langle (B^*)^n A^n \xi,\eta\rangle\vert & = \vert  \langle A^n \xi , B^n \eta\rangle\vert \leq \vert\langle  p_AA^n \xi , p_B B^n \eta\rangle\vert + \varepsilon\\
& = \vert \langle  p_B p_AA^n \xi , B^n \eta\rangle\vert + \varepsilon = \varepsilon.
\end{align*}
This proves the lemma.
\end{proof}

We now  give an easy criterion for proving the weak convergence of the rotations.
Consider $x\in B(\fH)\cap \{A,B\}'$ an operator acting on $\fH$ which commutes with $A$ and $B$.
For any tree $t$ with $n$ leaves we define $x_t\in B(\fH_t)$ by $x_t(t,\xi_1,\cdots,\xi_n):= (t,x\xi_1,\cdots,x\xi_n)$.
Note that since $x$ commutes with $A$ and $B$ we have that $\Phi(f)\circ x_t = x_{ft} \circ \Phi(f)$ for any forest $f$ with $n$ roots.
This implies that $(x_t)_t$ densely defines a map $[x]$ on the direct limit $\scrH$ and since
$\Vert x_t\Vert=\Vert x\Vert$ for any $t$ we have that $[x]$ is a bounded operator.

\begin{proposition}
Consider a Pythagorean couple $(A,B)$ acting on $\fH$ and its associated representation $(\pi,\scrH)$.
Assume that there exists $x\in B(\fH)$ which commutes with $A,B$ such that  $\langle \pi(r_n)\xi,\eta\rangle$ converges to $\langle x\xi,\eta\rangle$ for any $\xi,\eta\in \fH\subseteq \scrH$.
Then  $\pi(r_n)$ converges to $[x]$ in $B(\scrH)$ for the weak operator topology.
\end{proposition}

\begin{proof}
Following notations of Section \ref{sec:coefficients}, if $I$ is a standard dyadic interval, we denote by $\cA_I$ the operator equal to a product of $A$ and $B$ such that $\cA_{[0,1]}=\id, \cA_{[0,1/2]}=A, \cA_{[1/2,1]}=B, \cA_{[1,1/4]}=AA, \cA_{[1/4,1/2]}= BA, \cA_{[1/2,3/4]}=AB,$ etc.
Put $I_{m,k}:=[ \frac{k-1}{2^m} , \frac{ k }{2^m} ]$ and $\cA_{m,k}:= \cA_{I_{m,k}}$ for $m\geq 1$ and $1\leq k\leq 2^m$.
Fix some vectors $\xi,\eta\in \scrH$.
Let us show that $\lim_n\langle \pi(r_n)\xi,\eta\rangle=\langle [x]\xi,\eta\rangle$.
Consider the Hilbert spaces $\fH_{t_m},m\geq 1$ associated to the full binary tree $t_m$ with $2^m$ leaves and observe that they form an increasing union of subspaces of $\scrH$ whose union is dense.
Since $\pi(r_n), n\geq 1$ are unitary operators and thus are uniformly bounded we can assume by density that $\xi,\eta\in\fH_{t_m}$ for a certain $m\geq 1$.
Identify $\fH_{t_m}$ with $\fH^{2^m}$ and denote by $\xi_i$ and $\eta_i$ there $i$th component with the convention that $\xi_{i+2^m} = \xi_i.$
Fix $l\geq 1$ and consider now $\xi,\eta$ inside $\fH_{t_{m+l}}$ via the usual embedding $\fH_{t_m}\subset \fH_{t_{m+l}}.$
The $il+k$th component of $\eta$ in $\fH_{t_{m+l}}$ is then $\cA_{l , k}\eta_i$ for $1\leq i\leq 2^m$ and $1\leq k\leq 2^k.$
Let us apply the rotation $\pi(r_{m+l})$ on $\xi$ that shifts all of its components by one.
Hence, the $il+k$th component of $\pi(r_{m+l})\xi$ is 
$$\begin{cases} 
\cA_{l,k+1}\xi_i & \text{ if } 1\leq i \leq 2^m , \ 1\leq k \leq 2^l -1\\
\cA_{l,1}\xi_{i+1} & \text{ if } 1\leq i \leq 2^m , \  k=2^l \\
\end{cases}.$$
Observe that by considering the vectors $\xi_i,\eta_i$ as vectors in $\fH_{t_l}$ via the embedding $\phi(l):\fH\to \fH_{t_l}$ we obtain the formula:
\begin{equation}\label{equa:rot}
\langle\pi(r_l)\xi_i,\eta_i\rangle=\sum_{k=1}^{2^l}\langle\cA_{l,k+1}\xi',\cA_{l,k}\eta'\rangle.
\end{equation}
Therefore,
\begin{align*}
\langle \pi(r_{m+l})\xi,\eta\rangle & = \sum_{i=1}^{2^m}( \sum_{k=1}^{2^l-1} \langle \cA_{l,k+1}\xi_i, \cA_{l,k}\eta_i\rangle + \langle \cA_{l,1}\xi_i, \cA_{l,2^l} \eta_{i-1}\rangle)\\
& = \sum_{i=1}^{2^m}( \langle \pi(r_l)\xi_i,\eta_i\rangle + \langle \cA_{l,1}\xi_i , \cA_{l,2^l} (\eta_{i-1}- \eta_i)\rangle ) \text{ by \eqref{equa:rot}}\\
& = \sum_{i=1}^{2^m}( \langle \pi(r_l)\xi_i,\eta_i\rangle + \langle (B^l)^*A^l\xi_i, \eta_{i-1}- \eta_i\rangle ). \\
\end{align*}
Since $\xi_i,\eta_i$ belong to $\fH$ we have by assumption that $\langle \pi(r_l)\xi_i,\eta_i\rangle$ converges to $\langle x\xi_i,\eta_i\rangle$ and by Lemma \ref{lem:weak-limit} $\langle (B^{l})^*A^{l} \xi_i ,  \eta_{i-1} - \eta_i \rangle$ converges to $0$ for any $1\leq i\leq 2^m$.
Therefore, 
$$\lim_{l\to \infty} \langle \pi(r_{m+l})\xi,\eta\rangle = \sum_{i=1}^{2^m}\langle x\xi_i,\eta_i\rangle = \langle [x]\xi,\eta\rangle.$$
This finishes the proof.
\end{proof}

We deduce an equation that  the limit of $\pi(r_n)$ satisfies when we have some additional commutation assumptions.

\begin{corollary}
Suppose there exists $x\in B(\fH)\cap \{A,B\}'$ such that $\lim_n\langle\pi(r_n)\xi,\eta\rangle=\langle x\xi,\eta\rangle$ for any $\xi,\eta\in \fH$ and further assume that $A^*B$ commutes with $A,B,A^*,$ and $B^*$.
Then the sequence $(\pi(r_n))_n$ converges for the weak operator topology to $[x]$ and $x$ satisfies the following equation
$$x(\id-B^*A) = A^*B.$$
\end{corollary}

\begin{proof}
Since $A^*B$ commutes with $A,B,A^*,B^*$ we can define the operators $[A^*B]$ and $[B^*A]$ in $B(\scrH)$.
Fix $n\geq 1$, consider $\xi,\eta\in\fH$ and write $\xi_i:=\cA_{n,i}\xi,\eta_i:=\cA_{n,i}\eta$ for any $1\leq i\leq 2^n$ such that $\xi = (t_n,\xi_1, \cdots , \xi_{2^n})$ and $\eta = (t_n,\eta_1, \cdots , \eta_{2^n})$.
Observe that 
\begin{align*}
\langle \pi(r_{n+1}) \xi , \eta \rangle & = \sum_{i=1}^{2^n} \langle B\xi_i , A\eta_i \rangle + \sum_{j=1}^{2^n} \langle A\xi_{j+1} , B \eta_j \rangle\\
& = \langle [A^*B]\xi , \eta \rangle + \langle \pi(r_n) [B^*A] \xi , \eta \rangle\\
& = \langle A^*B\xi , \eta \rangle + \langle \pi(r_n) B^*A \xi , \eta \rangle.
\end{align*}
Taking the limit in $n$ we obtain $\langle x\xi , \eta\rangle = \langle A^*B \xi + x B^*A \xi , \eta\rangle$ for any $\xi,\eta\in \fH$ which proves the lemma.
\end{proof}

We will see that in most examples the sequence of rotations $\pi(r_n)$ does not tend to the identity and hence the action by the group generated by the rotations inside $T$ cannot be continuously extended to the rotation group of the circle.

\section{General properties of Pythagorean representations}
In this section we consider a pair $(A,B)$ of operators acting on the Hilbert space $\fH$ and satisfying the Pythagorean equation.
Denote by $(\pi,\scrH)$ the associated representation of Thompson's group $F$ or $T$.
We give some general properties of $\pi$.
Note that they can be generalised to representations of $F_k$ and $T_k$ for $k\geq 3$.

\subsection{Behavior of coefficients at infinity}\mbox{  }

We start by showing that a nontrivial coefficient of a Pythagorean representation does not tend to zero at infinity.

\begin{proposition}\label{prop:coef}
Consider two vectors $\xi,\eta\in\scrH$ and the associated coefficient $\varphi(g)=\langle \pi(g)\xi,\eta\rangle, g\in F.$
Then 
$$\limsup_{g\to\infty}|\varphi(g)|=\sup_{h\in F}|\varphi(h)|.$$
In particular, $(\pi,\scrH)$ does not produce any nontrivial positive definite function or any nontrivial coefficient  tending to zero at infinity in $F$ (and thus in $T$).
\end{proposition}
\begin{proof}
Consider  unit vectors $\xi,\eta\in \scrH$ and the associated coefficient $\varphi:F\to\C, g\mapsto \langle \pi(g)\xi,\eta\rangle.$
If $\alpha:=\sup_{h\in F}|\varphi(h)|$ is equal to zero, then the proof is trivial.
We assume that $\alpha$ is nonzero.
Suppose that there exists $\delta>0$ such that $\limsup_{g\to\infty} |\varphi(g)| < \alpha-\delta$.
Set $\varep := \delta/7.$
By changing $\xi$ to some $\pi(h)\xi$ we can assume that $|\varphi(e)|>\alpha-\varep$.
Recall that $t_n$ is the full binary tree with $2^n$ leaves.
By density there exists $n\geq 1$ and unit vectors $\xi',\eta'$ in the space $\fH_{t_n}$ such that $\Vert \xi - \xi' \Vert,\Vert \eta-\eta'\Vert<\varep.$
Since $\fH_{t_n} \subset \fH_{ t_{ n+1 } } $ we can choose a large $n$ satisfying $2^{ -n/2 } <  \varep $.
We define the coefficient  $\varphi'(g):=\langle \pi(g)\xi',\eta'\rangle, g\in F$ and observe that $| |\varphi(g) | - | \varphi'(g) | | < 2\varep, \forall g\in F.$ 
By assumption, there exists a subset $L\subset F$ with finite complement such that $| \varphi(g) | <\alpha-\delta$ for any $g\in L$.
We denote by $\xi_i',\eta_i'$ the $i$th component of the vector $\xi'$ and $\eta'$ respectively.
If $u_i,v_i$ are trees, $g_i= \frac{u_i}{v_i}\in F, 1\leq i\leq 2^n$ we put $\un g=(g_1,\cdots,g_{2^n})$ and 
$$\phi(\un g)=\phi(g_1,\cdots,g_{2^n}):=\frac{ (u_1\bu \cdots \bu u_{2^n})\circ t_n }{ (v_1\bu \cdots \bu v_{2^n})\circ t_n }\in F.$$ 
Recall that $\bu$ denotes the horizontal concatenation of forests as defined in Section \ref{sec:def}.
Therefore, $(u_1\bu \cdots \bu u_{2^n})\circ t_n$ is the tree obtained by gluing the tree $u_i$ on top of the $i$th leaf of $t_n$ for any $i$.
Observe that $\phi(\un g)$ only depends on $\un g$ and does not depend on the choice of the trees $u_i,v_i, 1\leq i\leq 2^n$.  
Moreover, 
$$\langle\pi\circ\phi(\un g)\xi' , \eta' \rangle=\sum_{i=1}^{2^n} \langle\pi(g_i)\xi_i' , \eta_i' \rangle.$$
The group element $\phi(\un g)$ tends to infinity if at least one of the $g_i$ tends to infinity.
Hence there exists a subset $L'\subset F$ with finite complement such that $\phi(\un g)\in L$ if at least one of the $g_i$ is in $L'$.
Let $j$ be the index satisfying that $\Vert \xi_j'\Vert=\min_i \Vert \xi_i'\Vert$.
Note that $\Vert \xi_j'\Vert\leq 2^{-n/2} < \varep$ since $\xi'$ is a unit vector.
Fix $g_j\in L'$ and put $\un g\in F^{2^n}$ whose each entry is equal to the identity except the $j$th entry that is equal to $g_j$.
We have that 
\begin{align*}
\alpha-\delta & > | \varphi ( \phi( \un g ) ) | >  | \varphi' ( \phi( \un g ) ) | -2\varep =  | \sum_{i=1}^{2^n} \langle\pi(g_i)\xi_i' ,\eta_i' \rangle | -2\varep \\
& = | \langle \xi',\eta'\rangle - \langle \xi_j' , \eta_j' \rangle + \langle \pi(g_j)\xi_j' , \eta_j' \rangle | -2\varep \\
& \geq |\varphi'(e) | - |\langle \xi_j' , \eta_j' \rangle| - | \langle \pi(g_j)\xi_j' , \eta_j' \rangle | - 2\varep\\
& > \alpha-3\varep - 2\Vert\xi_j'\Vert  -2\varep > \alpha - 7\varep = \alpha-\delta
\end{align*}
a contradiction.
\end{proof}

\begin{remark}
Since any coefficient of the regular representation $\lambda_F$ of $F$ tends to zero at infinity we obtain that $\lambda_F$ does not embed in $\pi$.
By a result of Dudko-Medynets we obtain that the representation $\pi$ does not admit any II$_1$ direct summands \cite{Dudko-Medynets14}.

Using tensor products instead of direct products we built families of representations having coefficients vanishing at infinity in \cite{Brothier-Jones18}. 
Proposition \ref{prop:coef} demonstrates how different those representations are from the Pythagorean one. 
\end{remark}

\subsection{Invariant vectors}\mbox{  }

A Pythagorean representation of $F$ can be trivial (see \ref{sec:AoneBzero}) and thus can have nonzero invariant vectors.
However, the next proposition shows that there are never nonzero $T$-invariant vectors.

\begin{proposition}
The Pythagorean representation $(\pi,\scrH)$ does not contain any nonzero $T$-invariant vectors.
\end{proposition}

\begin{proof}
Assume that $\xi\in\scrH$ is a $T$-invariant unit vector and fix $0<\varep'<1/3$.
By density, there exists $n\geq 1$ and a unit vector $\xi'=(t_n,\xi'_1,\cdots, \xi'_{2^n}) \in\fH_{t_n}$ such that $\Vert \xi - \xi'\Vert<\varep'$, where $t_n$ is the full binary planar tree with $2^n$ leaves.
Let $r_n\in T$ be the rotation of angle $2^{-n}$.
Since $\xi$ is $T$-invariant, we have that $\Vert \pi(r_n)^i\xi' - \xi' \Vert<2\varep'$ for any $i\geq 0$.
Set $S(i):=\sum_{k=1}^{2^n} \Vert \xi'_{i+k}-\xi'_i\Vert^2$, where we take the convention that $\xi'_{2^n + j} =\xi'_j$ for any $j$ and observe that 
$$\sum_{i=1}^{2^n} S(i)=\sum_{j=1}^{2^n} \Vert \pi(r_n)^j\xi'-\xi'\Vert^2<2^n.4\varep'^2.$$
Therefore, there exists $j$ such that $S(j)<4\varep'^2$.
Define the vector $\xi''\in\fH_{t_n}$ such that each entry is equal to $\xi'_j$.
We have $\Vert \xi' - \xi''\Vert^2 = S(j) <4\varep'^2$ which implies that $\Vert \xi -\xi''\Vert<3\varep'$ and $\Vert \pi(g)\xi'' - \xi''\Vert<6\varep'$ for any $g\in T.$
Let us show that $\xi'_j$ is close to being invariant under the action of the group $T$.

Consider an element $g=\frac{(a,k)}{(b,0)}\in T$ where $a,b$ are trees and $k\geq 0$.
We put $g_n:=\frac{( (a)_n\circ t_n, k ) }{ ( (b)_n\circ t_n, 0 ) }$ where $(a)_n$ is the forest with $2^n$ roots, each of whose tree is equal to $a$.
We have that 
$$ \langle \pi( g_n ) \xi'',\xi'' \rangle = 2^n \langle \pi(g) \xi'_j , \xi'_j \rangle.$$
Moreover, we have that $1+3\varep'>2^{n/2}\Vert \xi'_j\Vert = \Vert \xi''\Vert > 1 - 3\varep'>0$.
We obtain that $$\frac{| \langle \pi(g) \xi_j' , \xi_j'\rangle |}{\Vert \xi_j'\Vert^2} = \frac{ | \langle \pi( g_n ) \xi'',\xi'' \rangle | }{\Vert \xi''\Vert^2} > 
\frac{ \Vert \xi''\Vert - \Vert \pi(g_n)\xi'' - \xi''\Vert } { \Vert \xi''\Vert } > \frac{1-9\varep'}{1-3\varep'}.$$
Since $\varep'$ was arbitrary, we obtain that for any $\varep>0$ there exists a unit vector $\zeta$ in the small Hilbert space $\fH$ such that  $ \Vert \pi(g)\zeta - \zeta \Vert<\varep$ for any $g\in T$.
Fix such a $\zeta$ for $0<\varep<1.$
Consider the trees $a:= f_1\circ f_1$ and $b:=f_2\circ f_1$.
Put $g:=\frac{a}{b}\in F$ and $h=\frac{(a,1)}{(b,0)}\in T$ and observe that
\begin{align}\label{equa:A}
\Vert \pi(g)\zeta - \zeta\Vert^2 = \Vert A^2\zeta - A\zeta \Vert^2 + \Vert BA\zeta - AB\zeta \Vert^2 + \Vert B\zeta - B^2\zeta \Vert^2 & < \varep^2 ;\\ \label{equa:B}
\Vert \pi(h)\zeta - \zeta\Vert^2 = \Vert A^2\zeta - AB\zeta \Vert^2 + \Vert BA\zeta - B^2\zeta \Vert^2 + \Vert B\zeta - A\zeta \Vert^2 & < \varep^2.
\end{align}
Put $\eta:=A\zeta$ and observe that $\Vert A\eta - \eta\Vert<\varep$.
Moreover, $1=\Vert \zeta\Vert \leq \Vert A^*A\zeta\Vert + \Vert B^*B\zeta\Vert \leq \Vert A\zeta\Vert + \Vert B^*(B\zeta - A\zeta)\Vert + \Vert B^*A\zeta\Vert < 2\Vert \eta\Vert +\varep$ and thus $\Vert \eta\Vert >\frac{1-\varep}{2}>0$. 
The inequalities \eqref{equa:A}, \eqref{equa:B} imply that $\Vert A\eta - \eta\Vert<\varep$ and $\Vert B\eta - \eta\Vert<3\varep$.
Therefore,
\begin{align*}
\Vert A^*A\eta - \eta \Vert^2 & = \Vert A^*A\eta\Vert^2 + \Vert \eta\Vert^2 - 2\Vert A\eta\Vert^2\\
& \leq \Vert \eta \Vert^2 - \Vert A\eta\Vert^2 \leq 2 \Vert \eta\Vert ( \Vert \eta\Vert - \Vert A\eta\Vert ) \text{ since } \Vert A\Vert\leq 1\\
& \leq 2  \Vert \eta \Vert  \Vert A\eta - \eta\Vert < 2\varep  \Vert \eta \Vert.
\end{align*}
A similar argument combined with the Pythagorean equality gives that 
$$\Vert A^*A\eta\Vert^2=\Vert B^*B\eta -\eta\Vert^2<6\varep\Vert \eta\Vert.$$
Since $\eta=A^*A\eta+B^*B\eta$ we obtain that $\Vert\eta\Vert \leq (\sqrt{2}+ \sqrt{6})\sqrt{\varep\Vert \eta\Vert}$, a contradiction since $\eta\neq 0$ and $\varep$ is arbitrary small.
Therefore, $\pi$ does not have any nonzero $T$-invariant vectors.
\end{proof}

\begin{question}
Is there a Pythagorean representation that admits an almost $T$-invariant vector?
\end{question}

\section{Solutions of the Pythagorean equation}
 
 We now begin to investigate particular Pythagorean representations. 
 \begin{definition}
Let $\rho_{A,B}^\Omega$ (or just $\rho$ when there is no ambiguity)
be the subrepresentation of $\pi_{A,B}$ spanned by the orbit of $\Omega$ under $F$ (or $T$).
 \end{definition}
 
 We remind the reader that, if a unitary representation $\sigma$ admits a cyclic unit vector $\xi$ such that $|\langle \sigma(g)\xi,\xi\rangle|=1$ or $0$, then $\sigma$ is
 the representation \emph{induced} from the character $\chi$ of the subgroup $H$ for which $\langle \sigma(h)\xi,\xi\rangle=\chi(h),\  \forall h\in H$.

Note that the construction of Pythagorean representations behaves well under direct sums.
Indeed, if $(A_i,1\leq 1\leq n)$ and $(A_i', 1\leq i\leq n)$ are two families of operators satisfying the Pythagorean equation, then $(A_i'':=A_i\oplus A_i', 1\leq 1\leq n)$ satisfies it as well and we have that $\pi_{A''}\simeq \pi_{A}\oplus\pi_{A'}$.
One has a similar statement for infinite direct sums and direct integrals of operators.
For this reason we will consider examples that cannot be trivially decomposed into a direct sum.
 
 \subsection{The simplest case: $A=1, B=0$}\label{sec:AoneBzero}\mbox{  }

 The simplest solution of the Pythagorean equation is no doubt
 when $A$ or $B$ is zero and $\fH=\C$. In this case it is clear from the formula
\ref{formula} that $\langle\pi_{A,B}(g)\Omega,\Omega\rangle=1$ for
$g\in F$ since only the path that stays on the left contributes to
the sum. Thus $\rho_{A,B}^\Omega$ is trivial.

Perhaps surprisingly the representation $\rho_{A,B}$ of $T$ is not
trivial. For any path from bottom to top for a pair of trees in $T$ but not in $F$ must pass through an edge labelled $A$ and one labelled
$B$ or $B^*$. So by our remark, the representation of $T$ is the one induced from the trivial representation of  $F$.

\subsection{The usual action: $A=B=\frac{1}{\sqrt 2}$}\mbox{  }

Recall that $T$ acts by piecewise linear homeomorphisms on the circle $\bfS$ by transformations which are differentiable at all but a finite number of points, and the slopes are powers of $2$.

\begin{proposition} 
When $A=B=\frac{1}{\sqrt 2}$ and $\fH=\C$ there is an isomorphism from $\scrH$ to $L^2(\bfS)$ intertwining  the action $\pi_{A,B}$ of $T$ (hence $F$ by restriction) with the action on $L^2(\bfS)$ given by: $$g^{-1}\cdot f(z)=\sqrt {g'(z)} f(g(z)),$$
for any $g\in T, f\in L^2(\bfS),$ and almost every $z\in \bfS.$
\end{proposition}

\begin{proof}
Given $t\in\cF(1,n)$ we put $\cI_t$ the set of standard dyadic intervals corresponding to the leaves of $t$ and write $(t,\oplus_{I\in\cI_t}\xi_I)$ the elements of $\fH_t$.
Consider the map 
$$V_t:\fH_t\to L^2(\bfS), (t,\oplus_{I\in\cI_t}\xi_I)\mapsto \sum_{I\in\cI_t} \frac{\chi_I \xi_I}{\sqrt{\Leb(I)}},$$
where $\chi_I$ is the characteristic function of the interval $I$ and $\Leb$ is the Lebesgue measure.
It is easy to see that $V_t$ is an isometry and that $V_s = V_t\circ\iota^s_t$ if $t\leq s$ as trees.
Therefore, the system of maps $(V_t,t)$ defines an isometry $V$ from $\scrH$ to $L^2(\bfS)$.
This map has dense image by Stone-Weierstrass theorem and thus is surjective since the range of an isometry is closed.
One can check that the unitary operator $V$ intertwines $\pi_{A,B}$ and the action on the circle described above.
\end{proof}

Interestingly, multiplying both $A$ and $B$ by a fixed $\omega \in \C$ with $|\omega|=1$ actually changes the representation non-trivially. 
To be precise, the formula for $g$ acting on $L^2(\bfS)$ becomes 
$$g^{-1}\cdot f(z)=\omega^{-\log_2g'(z)}\sqrt {g'(z)} f(g(z)).$$
This can be proved by defining the family of isometries
$$V^\omega_t:\fH_t\to L^2(\bfS), (t,\oplus_{I\in\cI_t}\xi_I)\mapsto \sum_{I\in\cI_t} \frac{ \chi_I \xi_I \omega^{\log_2(\Leb(I))} }{\sqrt{ \Leb(I)}  }$$
and mimicking the proof above.
Note that, these ways of twisting the trivial bundle over $\bfS$ are not possible for all of $\Diff(\bfS)$.

\begin{remark}
Note that the rotation subgroup of $T$ acts by rotations so extends to all of the rotation group $\Rot(\bfS)$, and the representation on all of $T$ has strong continuity properties, in stark contrast to the representations of \cite{Jones16-Thompson}.
In general, it is a very difficult task to compute limits of rotations when the angle tends to zero especially when dealing with representations built with tensor products. 
Interestingly, this computation is more feasible in the Pythagorean context. As we will see in coming examples the limit of rotations will generically not tend to the identity thus excluding the possibility to extend the action to $\Rot(\bfS)$.
Is there a condition assuring the continuity of the rotations for a Pythagorean representation?
\end{remark}

\subsection{Arbitrary sums of real scalars: $A=\cos \theta, B=\sin \theta$}\mbox{ }

Here $\fH=\C$ and $\Omega$ is the number $1$. 
These Pythagorean representations interpolate between the last two.
We want to prove that for any $\theta$ they have no fixed vectors.
In the next section we will provide explicit formulae of limits of rotations for any choice of \emph{complex} scalars $A,B$.

Abbreviate $\cos \theta $ and $\sin \theta$ to $c$ and $s$ respectively.
Let $t_n$ be the full binary planar tree with $2^n$ leaves and each branch of length $n$. 
Now let $g_n$ be the elements of $F$ given by the pairs of trees consisting of $t_n$ on top and on bottom and $2^n$ copies of a simple pair of trees (the choice of which is not important) joining the leaves of the top to the bottom, thus:
$g_n=$  \vpic {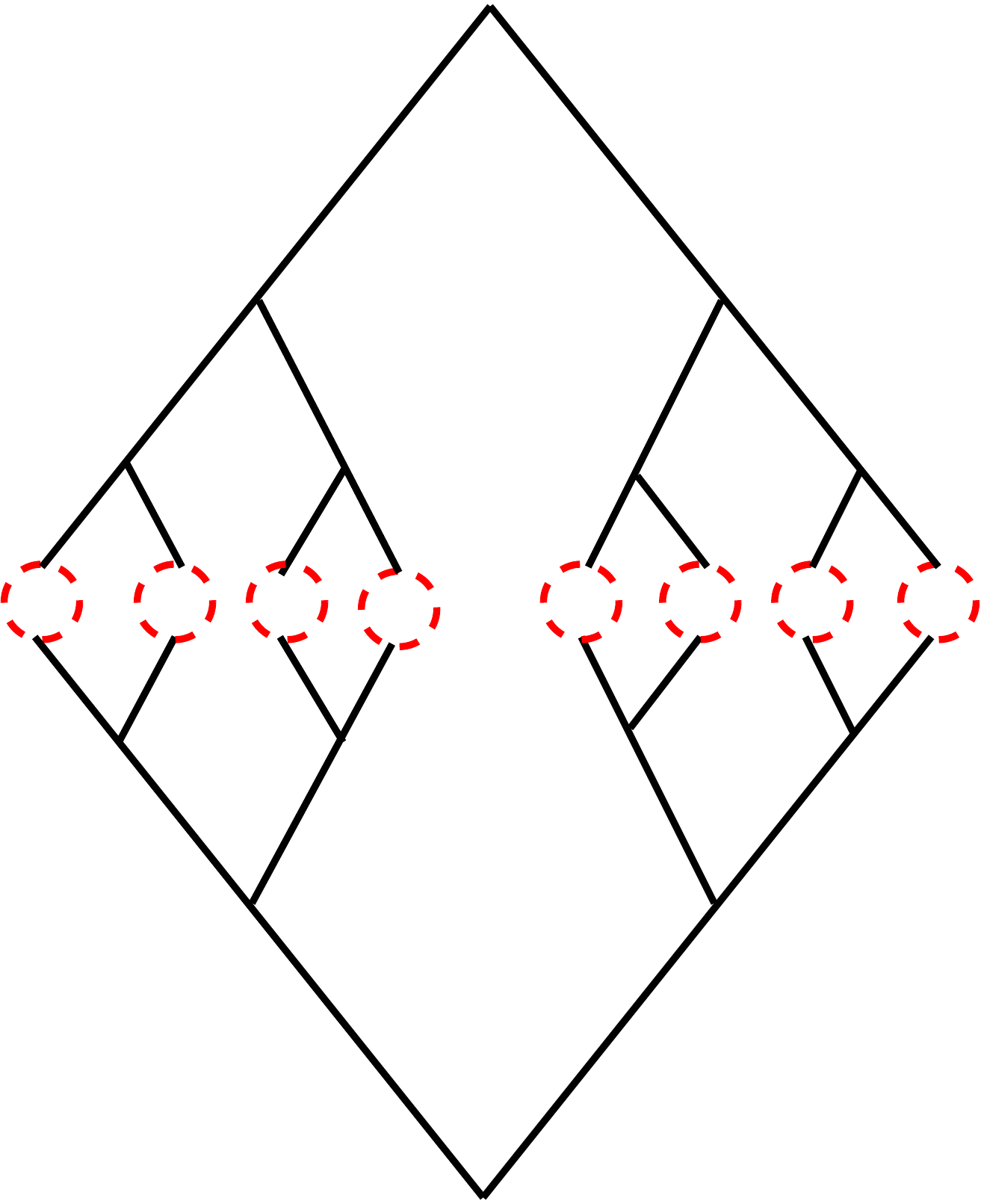} {1in} 
where each dotted circle is filled with a copy of the pair of trees: 
\vpic {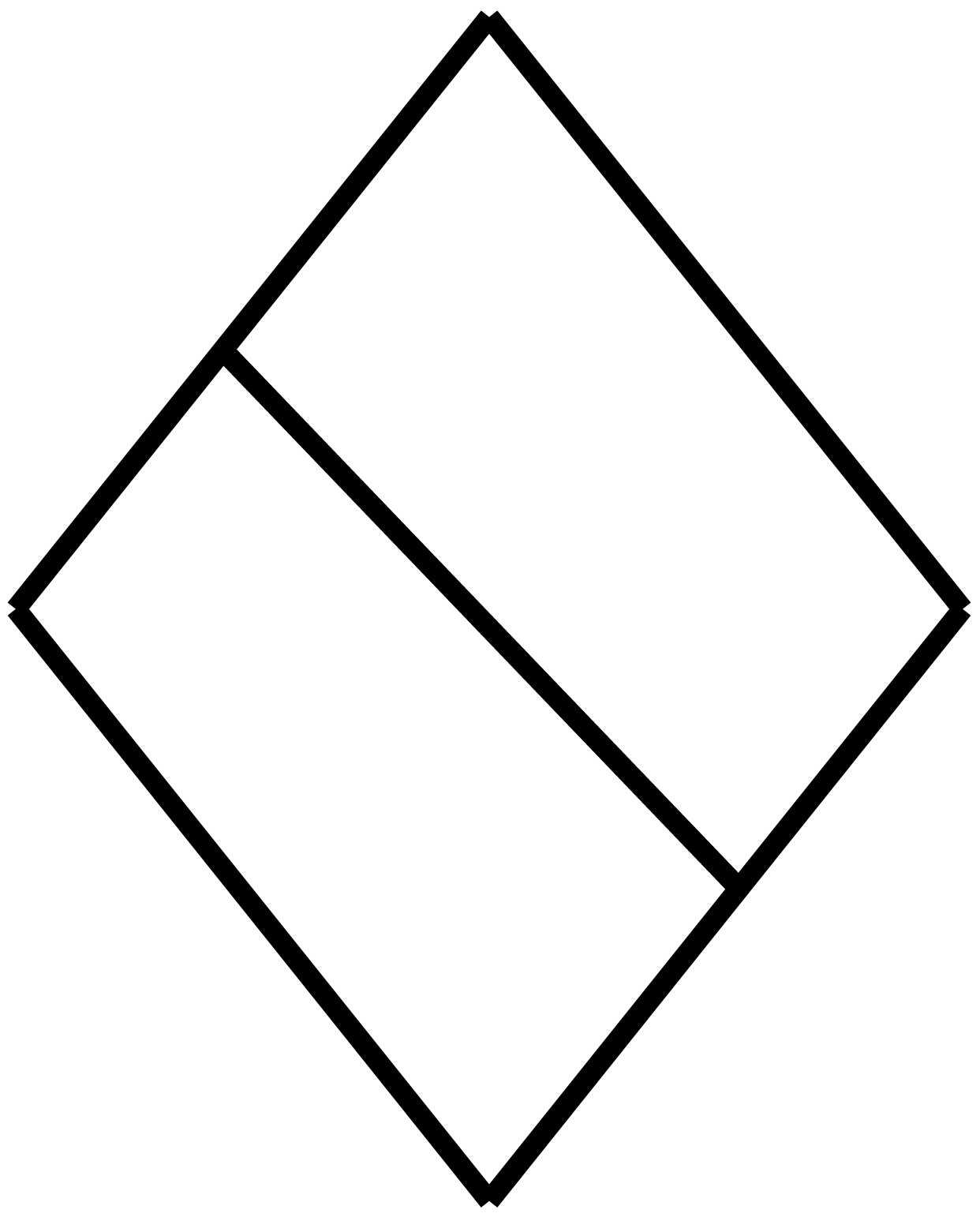} {0.3 in}

As a preliminary calculation let us figure out $\langle \pi(g_n)\Omega,\Omega\rangle$.
By the formula \ref{formula} the copies of the small pair of trees can all be removed if one multiplies by $c^3+s^2c^2+ s^3$. 
Once removed one sees simply $\langle \Omega, \Omega\rangle=1$. 
So for every $n$, $\langle \pi(g_n)\Omega,\Omega\rangle=c^3+s^2c^2+ s^3$.

We now proceed to show that the weak limit of the $\pi(g_n)$ on $\scrH$ is actually $c^3+s^2c^2+ s^3$. 
For this it suffices by boundedness of the $\pi(g_n)$ to show that 
$\lim_{n\rightarrow \infty} \langle \pi(g_n)\xi,\eta\rangle=(c^3+s^2c^2+ s^3)\langle \xi,\eta\rangle$
for every $\xi$ and $\eta$ in $(t_k,\oplus_1^{2^k} \mathfrak H)$ for all $k$. 
So let $( t_k,v)$ and $( t_k,w)$ be given with $v$ and $w$ in $\oplus_1^{2^k} \mathfrak H$. 
It is convenient to represent $v$ and $w$ as rectangular boxes which assign the entries of the vector to each of $2^k$ edges emanating from the top of the box. 
Forests can then be attached to the emanating edges to represent the action of  the category of forests on the vectors.
For instance if $v=(v_1,v_2)\in \C\oplus \C$, the following diagram represents the element $(cv_1,csv_1,s^2v_1,v_2)\in \oplus_1^4 \C$: \vpic {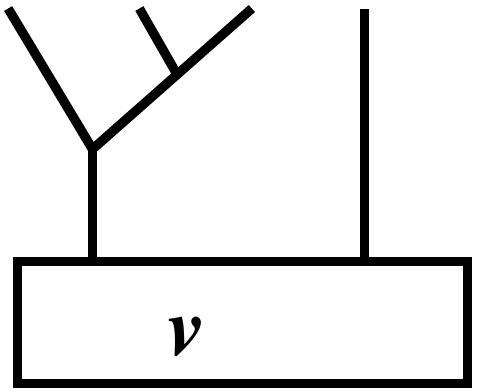} {.6in}

So suppose $v$ and $w$ are given in $\oplus_1^{2^k} \fH$ and $n$ is significantly larger than $k$.  
Then $\xi:=(t_k,v)$ and $\eta:=(t_k,w)$ are equal in $\scrH$ to pairs $(t_n,\hat v)$ and $( t_n,\hat w)$ where  $\hat v$ and $\hat w$  are obtained from $v$ and $w$ by attaching $2^k$ copies of $ t_{n-k}$ to the edges emanating from the boxes containing $v$ and $w$ thus:
\vpic {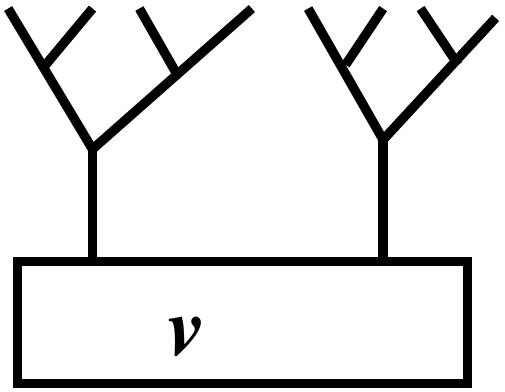} {.6in}
where for this illustration, $k=1$ and $n=3$.
If $s$ is a tree, we put $(s)_m$ the forest with $2^m$ trees all equal to $s$.
Consider the trees $a,b\in\cF(1,3)$ such that $g_n=\frac{(a)_n\circ t_n}{(b)_n\circ t_n}$.
Then 
\begin{align*}
\langle \pi(g_n)\xi , \eta \rangle & = \langle \frac{(a)_n\circ t_n}{(b)_n\circ t_n} \frac{t_n}{\hat v} , \frac{t_n}{\hat w}\rangle
 = \langle \Phi((b)_n)\hat v, \Phi((a)_n)\hat w \rangle\\
& = \langle \Phi(b) \Omega , \Phi(a) \Omega \rangle  \langle \hat v , \hat w\rangle
 = (c^3+s^2c^2+ s^3)   \langle \hat v , \hat w\rangle\\
& = (c^3+s^2c^2+ s^3)   \langle \xi , \eta\rangle.\\
\end{align*}
We obtain that the sequence $(\pi(g_n)_n)$ tends to $(c^3+s^2c^2+ s^3)\id$ for the weak operator topology.
In particular, there are no nonzero $F$-invariant vectors.

\subsection{Arbitrary scalars: computation of limits of rotations}\mbox{ }

Consider the one dimensional Hilbert space $\fH:=\C$ and the vacuum $\Omega=1$.
We consider a pair of complex numbers $A=a,B=b\in \C$ satisfying $| a |^2 + | b|^2=1$.
Let $(\pi,\scrH)$ be the associated representation.
Since $\pi$ is somewhat trivial when $a$ or $b$ is zero we assume that they are both nonzero and thus $|a|,|b|<1$.
We are interested in computing limits of rotations.
Let $r_n\in T$ be the rotation by $2^{-n}$, $r^j_n$ its $j$th power and set $\omega:=a \bar b$.

\begin{proposition}
For any $j\geq 1$ the sequence of operators $(\pi(r^j_n))_n$ converges for the weak operator topology to a scalar $x_j$ times the identity $\id$.
The family of scalars $(x_j)_j$ satisfy the following relations:
$$x_1 = \frac{\bar \omega}{1-\omega} , \ x_{2k} = x_k \text{ and } x_{2k+1} = \bar\omega x_k + \omega x_{j+1}, \ \forall k\geq 1.$$
\end{proposition}
\begin{proof}
We give a proof which does not use results of Section \ref{sec:rotation}.
Let us prove by induction on $j\geq 1$ that the sequence $(\pi(r_n^j))_n$ converges for the weak operator topology to $x_j$ (times the identity). 
Consider $\xi,\eta\in\fH_{t_n}$ with $i$th component written $\xi_i$ and $\eta_i$ respectively and put $\xi_{2^n+i}:=\xi_i$ for any $i$.
Observe that 
\begin{align*}
\langle \pi(r_{n+1})\xi , \eta \rangle & = \sum_{i=1}^{2^n} \langle B\xi_{i} , A\eta_i\rangle + \sum_{j=1}^{2^n} \langle A\xi_{j+1} , B\eta_j \rangle\\
& = \sum_{i=1}^{2^n} \bar a b \xi_{i} \bar\eta_i + \sum_{j=1}^{2^n} a\bar b \xi_{j+1}\bar\eta_j\\
& = \bar \omega \langle \xi , \eta \rangle + \omega \langle \pi(r_{n})\xi , \eta \rangle.
\end{align*}
Put $y_n:= \langle \pi(r_n)\xi,\eta\rangle -x_1\langle \xi,\eta\rangle$ and observe that $y_{n+1} = \omega y_n$ with $|\omega|<1$.
Therefore, $(y_n)_n$ tends to zero and $\langle \pi(r_n)\xi,\eta\rangle$ tends to $x_1\langle\xi,\eta\rangle$.
By density we obtain that $(\pi(r_n))_n$ converges (for the weak operator topology) to $x_1$.
Remark that if $(\pi(r_n^j))_n$ converges to $x_j$, then $(\pi(r_n^{2j}))_n$ converges to the same limit since $r_n^{2j} = r_{n-1}^j$ and in particular $(\pi(r_n^2))_n$ converges to $x_1$.

Consider $k\geq 1$ and assume that $(\pi(r_n^j))_n$ converges to $x_j$ for any $1\leq j\leq 2k$.
Fix some vectors $\xi,\eta\in\fH_{t_n}$ and observe that 
\begin{align*}
\langle \pi(r_{n+1}^{2k+1})\xi , \eta \rangle & = \sum_{i=1}^{2^n} \langle B\xi_{i+k} , A\eta_i\rangle + \sum_{j=1}^{2^n} \langle A\xi_{j+k+1} , B\eta_j \rangle\\
& = \sum_{i=1}^{2^n} \bar a b \xi_{i+k} \bar\eta_i + \sum_{j=1}^{2^n} a\bar b \xi_{j+k+1}\bar\eta_j\\
& = \bar\omega  \langle \pi(r_{n}^k)\xi , \eta \rangle + \omega \langle \pi(r_{n}^{k+1})\xi , \eta \rangle. 
\end{align*}
The induction hypothesis implies that the right hand side converges to $(\bar \omega x_k + \omega x_{k+1})\langle \xi , \eta\rangle$ and we obtain by density that $(\pi(r_n^{2k+1}))_n$ converges to $\bar \omega x_k + \omega x_{k+1}$.
This finishes the proof of the proposition.
\end{proof}

The formulae $x_{2k}=x_k$ and $x_{2k+1} = \bar\omega x_k + \omega x_{k+1}$ give us an algorithm for computing $x_j$.
Given $j\geq 1$ we construct a decorated binary tree which will determine $x_j$.
Start with the trivial tree and its root decorated by $j$.
If $j$ is a power of two, then we keep this trivial decorated tree and $x_j=x_1=\frac{\bar \omega}{1-\omega}$.
If not, consider $n$ such that $\frac{j}{2^n}$ is an odd integer equal to $2k+1$ for a certain $k> 0$.
We add a decorated trivalent node with $j$ on the bottom, $k$ on the top left, and $k+1$ on the top right.
Then we continue the same construction applied to $k$ and $k+1$ until we reach $1$ everywhere.
We obtain a binary tree $t(j)$.
For any leaf $\ell$ of $t(j)$ we associate the scalar $c_\ell:=\bar\omega^l\omega^r$ where $l$ (resp. $r$) is the number of left turns (resp. right turns) in the (geodesic) path from the root of $t(j)$ to the leaf $\ell$.
We obtain the formula $$x_j = \sum_{\ell \text{ a leaf of } t(j) } c_\ell x_1.$$
For example, if $j=13$, then
$$t(13)=\xthirteen$$ 
and
$$x_{13} = (\bar\omega^2 +\bar \omega\omega + \bar\omega^2\omega + \bar\omega \omega^2 + \omega^2)x_1.
$$

\subsection{Unitaries}\label{unitaries}\mbox{  }

If $u$ and $v$ are unitaries on some Hilbert space $\fH$, the pair $(A=\frac{1}{\sqrt 2}u, B=\frac{1}{\sqrt 2}v)$
is a Pythagorean pair. We will fix this choice of $(A,B)$ for the remainder of this subsection. For simplicity we restrict our
calculations to $F$, leaving the case of $T$ to the reader.

\begin{theorem}\label{cocycle} 
There is a \emph{unitary cocycle} $F\times [0,1]\ni (g,x) \mapsto U_{g,x}$ on $F$ with coefficients in the unitary group of $\fH$ such that the representation $\pi_{A,B}$ on $\scrH$ is equivalent to the unitary representation $\sigma$ on $L^2([0,1],\fH)$ given by
$$\sigma_g(\xi)(x)=g'(x)^{-1/2}U(g,x)(\xi(g^{-1}x)).$$
 
\end{theorem}
\begin{proof}
We begin by identifying $\scrH$ with $L^2([0,1],\fH)$. The theorem will then follow simply by transporting the
action of $F$ on $\scrH$ to $L^2([0,1],\fH)$.

So let $t\in \mathfrak T$ be given with $n$ leaves, $\mathcal I=\{I\}$ be the corresponding standard dyadic partition so that each $I$ is
a leaf of $t$,  and let $\xi =\oplus_{I\in \mathcal I} \xi_I$ be an element of $\fH_t$.
We set $$\cR_{\cI}(\xi)=\sum_{I\in \cI} \frac{\cA_I^*(\xi_I)\chi_I}{\Leb(I)}$$ 
where $\chi_I$ is the characteristic function of $I$ and $\Leb$ the Lebesgue measure. 
(For the meaning of $\mathcal A_I$ see \ref{leafoperator}.) 

Since $\cA_I$ is $\sqrt{\Leb(I)}$ times a unitary operator, 
$$\int_0^1 ||\cR_{\cI}(\xi)(x)||^2 dx =\sum_{I\in \cI} \frac{\Leb(I)\times \Leb(I)}{\Leb(I)^2} ||\xi_I\chi_I||^2=||\xi||^2.$$
The next thing to show is that $\cR$ is compatible with the direct limit. 
Let $N$ be the $i$th interval of $\cI$ that we split in two as $N=K\cup L$ where $K$ is the first half and consider the refined partition $\cJ$ containing $K,L$ and all intervals of $\cI$ except $N$.
We need to show that 
\begin{equation}\label{equa:RIJ}\mathcal R_{\cJ}(\Phi(f_i)(\xi))=\cR_{\cI}(\xi).\end{equation}

But in 
$$\mathcal R_{\cJ}(\Phi(f_i)(\xi))=\sum_{M\in \mathcal J} \frac{\mathcal A_M^*(\Phi(f_i)(\xi)_M)\chi_M}{\Leb(M)}$$ 
all the terms are equal to those of $\cR_\cI(\xi)=\sum_{I\in \mathcal I} \frac{1}{\Leb(I)}\mathcal A_I^*(\xi_I)\chi_I$ except 
$$\frac{\mathcal A_K^*((\Phi(f_i)(\xi)_K)\chi_K}{\Leb(K)} \text{ and } \frac{\mathcal A_L^*((\Phi(f_i)(\xi)_L)\chi_L}{\Leb(L)}$$ 
which are 
$$\frac{\mathcal A_K^*(\frac{1}{\sqrt 2} u(\xi_{N}))\chi_K}{\Leb(K)} \text{ and } \frac{\mathcal A_L^*(\frac{1}{\sqrt 2} v(\xi_{N}))\chi_L}{\Leb(L)}$$ 
respectively.
However $\mathcal A_K=\frac{1}{\sqrt 2}u \cA_{N}$ and $\cA_L=\frac{1}{\sqrt 2}v \cA_{N}$ and $\Leb(K)=\Leb(L)=\frac{1}{2}\Leb(N)$ so that the sum of these two terms is $\frac{1}{\Leb(N)}\cA_{N}^*(\xi_N)  \chi_{N}$ implying \ref{equa:RIJ}.
 
Since any tree is a combination of some $f_i$ we obtain that the collection of $\cR_\cI$ define an isometry $\mathcal R$ from $\scrH$ to $L^2([0,1],\fH)$. 
It's image clearly contains the dense subspace spanned by functions of the form $\xi\chi_I$ for all $\xi\in \fH$ and standard dyadic intervals $I$. So $\mathcal R$ extends to a unitary.
  
To finish the proof we need to see how the action of $F$ on $\scrH$ transports under $\mathcal R$. So let $g=\frac{\cI}{\cJ} $ be  an element of $F$ and choose a vector $\xi\chi_J$ for some $J\in \cJ$. 
By definition, $\mathcal R^{-1}(\xi\chi_J)$ is the (equivalence class of the) element $\frac{\cJ}{\eta}$ where  $\eta\in \oplus _{M\in \cJ} \fH$ is zero except when $M=J$ for which it is $\cA_J(\xi)$. 
By construction $\pi(g)(\frac{\cJ}{\eta})=\frac{\cI}{\eta}$ where $g(J)=I\in \cI$. 
Hence $$\cR \pi(g) \cR^{-1}( {\xi \chi_J} ) = \frac{ \cA_{ g J }^* \cA_J( \xi )\chi_{ g  J  }  }{ \Leb(gJ) } .$$
So for $g=\frac{\cI}{\cJ} $ and $x\in J\subseteq [0,1]$ define the unitary $U(g,x)= \frac{\cA_{gJ}^*\cA_J}{\sqrt{\Leb(gJ) \Leb( J ) } } .$
Then
$$\mathcal R\pi(g)\mathcal R^{-1}({\xi \chi_J})=\frac{\sqrt{\Leb(J) } }{ \sqrt{ \Leb( gJ ) } }(U(g,x)\xi) \chi_{gJ}$$ and
$$(\mathcal R\pi(g)\mathcal R^{-1}({\xi}))(x)=g'(x)^{-1/2}(U(g,x)\xi)(g^{-1}x).$$                                                           
\end{proof}
 
The interesting issue is the continuity of the representation for some topology on $F$ or $T$ as a group of homeomorphisms.

\subsection{The free group}\mbox{ }

Let $a$ and $b$ freely generate the free group on two generators $\mathbb F_2$ and let $u_a$ and $u_b$ be
the corresponding unitaries of the left regular representation on $\ell^2(\mathbb F_2)$.
Then putting $A=\frac{1}{\sqrt 2}u_a$ and $B=\frac{1}{\sqrt 2}u_v$ we get a solution of the Pythagorean equation \ref{pythagoras} and thus a unitary representation $\pi$ of $F$. 
If we choose $\Omega$ to be the characteristic function of the identity of $\mathbb F_2$ then when we calculate the coefficient $\langle\pi(g)\Omega,\Omega\rangle$ the only paths on the pair of trees  which contribute to the sum are those in which the path on the top tree is exactly the reflection (in the horizontal like between the two trees)  of the bottom path. 
But this means that the intervals indexed by the top and bottom paths are the same, so the Thompson group element is the identity on that interval. 
Moreover, the factors of $\frac{1}{\sqrt 2}$ give precisely the Lebesgue measure $\Leb$ of that interval. 
We obtain the result:
\begin{proposition}
For any $g\in F$ we have the equality 
$$\langle\pi(g)\Omega,\Omega\rangle=\Leb( \{  \text{fixed points of } g \} ).$$
\end{proposition}

This positive definite function for arbitrary nonfree \emph{measure-preserving} actions of groups was investigated by Vershik in \cite{Vershik11-NonFree}.

Note that rotations have no fixed points and the images of $\Omega$ by two different rotations are orthogonal so that the rotation group by dyadic rationals acts as its regular representation on the span of the orbit of $\Omega$.

Consider the group $F$ with the metric $d(g,h):=\Leb(\{ x\in [0,1] \ : g(x)\neq h(x) \} )$.
This turns $F$ into a topological group.
The proposition of above shows that $\langle\pi(g)\Omega,\Omega\rangle=1-d(g,e)$ for any $g\in F$.
This implies that the cyclic unitary representation $\rho^\Omega_{A,B}$ of $F$ is continuous (for the strong operator topology).

\subsection{Canonical anticommutation relations: $B=A^*$}\label{sec:CAR}\mbox{ }

The operator of creation of a single fermion is the $2\times 2$ matrix $a=\left( \begin{matrix} 0&1\\ 0&0 \end{matrix} \right)$ which satisfies the canonical anticommutation relations (CAR) $aa^* + a^*a=1$ and $a^2=0$. 
So we may construct a Pythagorean representation with $\fH=\C^2$ and $A=a, B=a^*$.
If we let $\Omega$ be the vector $(0,1)\in \fH$ we see from the formula \ref{formula} and $a^2=(a^*)^2=0$ that $\langle \pi(g)\Omega,\Omega\rangle=0$ unless the pair of trees for $g$ has a ``zigzag'', i.e.~a path that goes from bottom to top starting to the left  then alternating turning zero and right. 
In this case $\langle \pi(g)\Omega,\Omega\rangle=1$. 
Thus the set $F_\Omega$ of such pairs of trees defines the subgroup of $F$ (and $T$) that fixes $\Omega$ and the Pythagorean representation is induced from it.

\begin{proposition}\label{prop:1/3} 
The group $F_\Omega<F$ just defined is the stabiliser of $\frac{1}{3}$ for the action of $F$ on $[0,1]$ and the cyclic representation $\rho$ generated by $\Omega$ is unitary equivalent to the quasi-regular representation $\lambda_{F/F_\Omega}:F\act\ell^2(F/F_\Omega)$.
\end{proposition}

\begin{proof} 
The set of intervals indexed by the vertices of 
a zigzag path on a single tree is precisely the set of standard dyadic intervals containing
$\frac{1}{3}$. Thus any element of $F$ stabilising $\frac{1}{3}$ must
take a zigzag interval into another. Parity considerations show that 
the two individual zigzags combine to form a full zigzag. Conversely if a pair of trees contains a zigzag then the element of $F$ maps a
 standard dyadic interval containing $\frac{1}{3}$ to another. If these intervals are the same it fixes $\frac{1}{3}$. If not, iterating either the transformation or its inverse produces shows that the transformation fixes $\frac{1}{3}$.
Since the vector state $g\mapsto \langle \pi(g)\Omega,\Omega\rangle$ is equal to the characteristic function of $F_\Omega$ we obtain that $\rho$ and $\lambda_{F/F_\Omega}$ are unitary equivalent.
\end{proof}
Following the same idea we can produce the quasi-regular representation $\la_{F/F_x}$ where $F_x$ is the stabiliser of any real $x\in [0,1]$ for the usual action $F\act [0,1]$.
Indeed, fix $x\in [0,1]$ and a ray in the infinite rooted binary tree (whose vertices are identified with standard dyadic intervals in the usual way) starting at the root such that $x$ is in each interval appearing in this path. 
Put $L\subset \N$ the set of $k$ such that the $k$th edge of the ray turns to the left and set similarly $R$ for the right turns.
Define the following operators $A,B\in B(\N)$ such as
\begin{align*}
A:&\delta_k\mapsto \delta_{k+1} \text{ if } k\in L \text{ and } 0 \text{ otherwise;}\\
B:&\delta_k\mapsto \delta_{k+1} \text{ if } k\in R \text{ and } 0 \text{ otherwise,}\\
\end{align*}
where $(\delta_n,n\in\N)$ is the canonical basis of $\N$.
They form a Pythagorean couple and we obtain that the cyclic component generated by $\delta_0$ is unitary equivalent to the quasi-regular representation $\la_{F/F_x}.$
We invite the reader to consult \cite{Golan-Sapir18} for results concerning the algebraic structure of this type of subgroups.

\subsection{$A$ and $B$ projections}\mbox{ }

Consider a Hilbert space $\fH$, a unit vector $\Omega\in\fH$, and two projections $A,B\in B(\fH)$ satisfying the Pythagorean identity.
Necessarily $A$ and $B$ commute and $B=\id-A$.
If $t$ is a tree, then the first (resp. last) component of $\Phi(t)\Omega$ is equal to $A\Omega$ (resp. $B(\Omega)$) and all the others are equal to zero.
Therefore $\langle\pi(g)\Omega,\Omega\rangle =1$ if $g\in F$ and zero otherwise meaning that $\rho$, as a representation of $T$, is induced from the trivial representation of $F$. 

\subsection{Actions for $F_3$: $\fH=\C$}\mbox{ }

Set $\fH=\C, A_1=1/\sqrt 2, A_2=0, A_3=1/\sqrt 2,$ and consider the associated Pythagorean representation of the Thompson's group $F_3$.
Let $\fC\subset[0,1]$ be the Cantor ternary set and write $F_\fC<F_3$ the stabiliser subgroup $F_\fC:=\{g\in F_3 : gc=c, \forall c\in\fC\}$.

\begin{proposition}\label{prop:cantor}
The stabiliser $F_\Omega$ of the vacuum vector $\Omega=1$ is equal to $F_\fC$.
\end{proposition}

\begin{proof}
In this proof we identify tryadic trees with standard tryadic partitions and leaves with standard tryadic (closed) intervals.
Consider a standard tryadic interval $I\subset [0,1]$ and observe that its interior intersects $\fC$ if and only if the path from the root to $I$ in the infinite rooted ternary tree has only left and right turns but no no-turns. 
In that situation, we say that $I$ "has no no-turns."
Note that if $I$ has at least one no-turn, then it intersects $\fC$ only at its extremities.
In particular, if $\cI$ is a standard tryadic partition of $[0,1]$, then $\fC$ is included in the union of the $I\in\cI$ such that $I$ doesn't have any no-turns.
Therefore, if $g=\frac{\cJ}{\cI}\in F_3$ with $\cI,\cJ$ standard tryadic partitions of $[0,1],$ then $g$ fixes each point of $\fC$ if and only if $g$ is the identity when restricted to an interval $I\in\cI$ without no-turns.
This means that the collections of intervals of $\cI$ and of $\cJ$ that have no no-turns are equals.
Observe that 
$$(\Phi(\cI)\Omega)_I = \begin{cases}
\sqrt 2^{\log_3(\Leb(I))} \text{ if $I$ has no no-turns }\\
0 \text{ otherwise }  
\end{cases}.$$
An easy induction on the number of intervals in $\cI$ together with this description of $\Phi(\cI)\Omega$ implies that $\Phi(\cI)\Omega=\Phi(\cJ)\Omega$ if and only if the collections of intervals of $\cI$ and of $\cJ$ that have no no-turns are equals implying $F_\Omega = F_\fC$.
\end{proof}
\subsection{$A_i$ all multiples of mutually noncommuting projections}\label{sec:projNC}\mbox{  }

Consider $\fH=\C^2$, the vector $\xi_1=(1,0)$ and its images under the rotations of angle $2\pi/3$ and $4\pi/3$: $\xi_2=(1/2,-\sqrt 3/2)$ and $\xi_3=(1/2,\sqrt 3/2).$
Consider $p_i$ the minimal orthogonal projection onto $\C\xi_i$ and observe that $p_1+p_2+p_3=3/2$ and that $p_1,p_2,p_3$ do not mutually commute.
We set $A_i:=\sqrt{2/3} p_i$ for $i=1,2,3$ which is a Pythagorean triple with associated functor $\Phi$ and unitary representation $(\pi,\scrH).$
Since the $p_i$ are minimal projections we have some nice reduction formulae for any words of them.
A routine computation gives us 
$$ p_i p_j p_i = \frac{1}{4} p_i \text{ and } p_i p_j p_k p_i = -\frac{1}{8} p_i \text{ for any } i\neq j\neq k\neq i.$$
Consider a tree $t$ and write $d(t,\ell)$ the distance from its $\ell$th leaf to its root.
We have the following formula for the $\ell$th component of the operator $\Phi(t)$:
$$(\Phi(t))_\ell = \sqrt{2/3}^{ d( t , \ell ) } \frac{ \varep }{ 2^a } p_i p_j p_k \text{ for some } \varep\in\{-1,+1\},\ a\geq 0, \ i,j,k\in \{1,2,3\}.$$
Consider the vacuum vector $\Omega:=\xi_1$ and assume that $\pi(g)\Omega=\Omega$ for a fraction $g=\frac{t}{s}$.
Since $\langle \pi(g)\Omega , \Omega\rangle = \langle \Phi(s)\Omega , \Phi(t)\Omega\rangle$ we have that $\Phi(t)\Omega=\Phi(s)\Omega$.
We have $\Omega=p_1\Omega$ and thus $p_1 (\Phi(t)\Omega)_\ell  = \sqrt{2/3}^{d(t,\ell)} \frac{\epsilon}{2^b}\Omega$ for some $\epsilon\in\{-1,+1\},\ b\geq 0.$
We obtain that $d(t,\ell)=d(s,\ell)$ for any leaf $\ell$ and thus $t=s$ implying that $g=e$ is trivial.

Not all vectors of $\C^2$ have a trivial stabiliser.
Consider the vector $\Omega=(0,1)$ that is orthogonal to $\xi_1$ and thus $A_1\Omega=0$. 
Let $t',s'$ be two distinct trees with the same number of leaves.
Construct the tree $t$ that is the composition of the unique ternary tree with three leaves to which we attach $t'$ on its first leaf and construct similarly $s$.
The fraction $g=\frac{t}{s}$ is a nontrivial element of $F_3$ that fixes the vector $\Omega$.

\subsection{Connes-Landi spheres}\label{sec:Connes-Landi}\mbox{  }

Connes and Landi introduced a family of noncommutative spheres which produces tuples of operators satisfying the Pythagorean equation \cite{Connes-Landi01}.
We gratefully thank Georges Skandalis for pointing out this very interesting class of examples.
Consider a $n\times n$  skew-symmetric matrix $\theta=(\theta_{kl})_{kl}$ and put $\rho_{kl}:=e^{ 2\pi i \theta_{kl} } $ for any $1\leq k,l\leq n$.
The noncommutative $(2n-1)$-dimensional sphere twisted by $\theta$ is the universal C*-algebra $C(S_\theta^{2n-1})$ generated by $n$ normal elements $A_1,\cdots,A_n$ satisfying the Pythagorean equality $\sum_{k=1}^n A_k^*A_k=\id$ and such that $A_kA_l=\rho_{kl} A_lA_k$ for any $1\leq k,l\leq n$.

We consider an easy case with $n=2$ and when $(A,B)$ satisfies the relation of the noncommutative torus as follows:
let $\fH=L^2(\bfS)$ be the Hilbert space of $L^2$-function on the circle and consider the operators $Af(z)=\frac{zf(z)}{\sqrt 2}, Bf(z)=\frac{f( \la z)}{\sqrt 2}$ for a fixed complex number $\la$ of modulus one and for any $f\in L^2(\bfS)$ and almost every $z\in \bfS$ \cite{Connes80-Calg,Rieffel81-NCT}.
Using the Fourier transform we obtain the following description: $\fH=\ell^2(\Z), A\delta_n=2^{-1/2} \delta_{n+1}, B\delta_n=2^{-1/2} \la^n\delta_n, \forall n\in\Z$.

Put $\Omega:=\delta_0$ and consider a tree $t$ with $n$ leaves.
Denote by $d(t,j)$ the distance between the root of $t$ and its $j$th leaf and $L(t,j)$ the number of left turns from the root to the $j$th leaf in $t$.
Observe that the $j$th component of the vector $\Phi(t)$ is equal to the following:
$$(\Phi(t)\Omega)_j=\frac{\la^{R(t,j)}}{2^{d(t,j)/2}}\delta_{L(t,j)},$$
where $R(t,j)$ is a natural number depending on $t$ and $j$.
This gives us a formula for computing coefficients.
If $s,t$ are trees with $n$ leaves and $g=\frac{t}{s}\in F$, then 
$$\langle \pi(g)\Omega,\Omega\rangle=\langle \Phi(s)\Omega , \Phi(t)\Omega\rangle = \sum_{j=1}^n \frac{\la^{R(s,j)-R(t,j)}}{2^{(d(s,j)+d(t,j))/2}} \langle\delta_{L(s,j)} , \delta_{L(t,j)}\rangle.$$
Note that if $g$ stabilises $\Omega$, then necessarily $d(t,j)=d(s,j)$ for any $j$ implying that $t=s$ and thus $g=e$.\\
{\bf Limit of the rotations.}
Since we are considering a noncommutative version of the torus we could hope that the rotation is continuous.
Unfortunately this is not the case.
Consider the full binary tree $t_n$ with $2^n$ leaves and $r_n\in T$ the one click rotation corresponding to the rotation by angle $2^{-n}$.
Consider the vacuum vector $\Omega=\delta_0$ and observe that 
\begin{align*}
| \langle \pi(r_n)\Omega,\Omega\rangle | & \leq  \sum_{j=1}^{2^n} | \langle (\Phi(t_n)\delta_0)_j , (\Phi(t_n)\delta_0)_{j+1}\rangle | \\
& = \frac{ | \{ 1\leq j \leq 2^n : L(t_n,j)=L(t_n,j+1) \} | }{2^n}.
\end{align*}
Note that if $\lambda=1$, then the previous inequality is in fact an equality.
The right hand side is equal to the probability of two consecutive words for the lexicographic order of length $n$ in an alphabet of two letters $a,b$ to have the same number of occurrences of the letter $a$.
Observe that two consecutive words of length $n$ have the same number of occurrences of $a$ if and only if they are as follows:
$x_1\cdots x_{n-2} ab$ and $x_1\cdots x_{n-2} ba$ with $x_1,\cdots,x_{n-2}\in\{a,b\}.$
This implies that $|\langle \pi(r_n)\Omega,\Omega\rangle|\leq 1/4$ for any $n$.
Therefore, the sequence of rotation of $(\pi(r_n))_n$ does not tend to the identity even for the weak operator topology.

When $\lambda=1$, then $B=\frac{1}{\sqrt 2}$ and thus $A^*B=A^*$ commutes with $A,B,A^*,B^*$.
The analysis of Section \ref{sec:rotation} implies that $\pi(r_n)$ tends for the weak operator topology to $[x]$ where $x=A^*B(1-B^*A)^{-1} = \frac{S^*}{2} + \frac{1}{4} + \sum_{m=1}^\infty \frac{S^m}{2^{m+2}}$ where $S$ is the shift unitary operator.
We find that $\lim_n\langle \pi(r_n)\Omega,\Omega\rangle = \frac{1}{4}$ for $\Omega=\delta_0$ which corroborates our previous computation.

\section{Cuntz algebra}
The Cuntz algebra $O_n$ is the universal C*-algebra with generators $S_1,\cdots, S_n$ subject to the relations: $\sum_{i=1}^n S_iS_i^*= \id$ and $S_i^* S_j=\delta_{i,j} \id$ for any $1\leq i,j\leq n$ where $\delta_{i,j}$ is the Kronecker delta, see \cite{Cuntz77}. 
In particular, $O_n$ is a quotient of the C*-algebra $P_n$ of Definition \ref{def:P_n} and thus any representation of $O_n$ defines a representation of $P_n.$

The next proposition shows that in fact any representation of $P_n$ can be dilated to a representation of $O_n$ and conversely a compression of a representation of $O_n$ by a well-behaved projection provides a representation of $P_n$.
We are grateful to Anna Marie Bohmann and Ruy Exel for enlightening discussions which led to this result, in particular the dilation
construction is an adaptation of a category theoretic construction of Bohmann.

\begin{proposition}\label{prop:Cuntz}
If $\alpha: P_n\to B(\fH)$ is a representation, then there exists a larger Hilbert space $\scrH\supset \fH$ and a representation $\beta: O_n\to B(\scrH)$ such that $p \beta(S_i) p = p \beta(S_i) = \alpha(A_i)^*$ for any $1\leq i \leq n$, where $p$ is the orthogonal projection onto $\fH.$

Conversely, if $\beta: O_n\to B(\scrH)$ is a representation and $p\in B(\scrH)$ a projection satisfying $p \beta(S_i) p = p \beta(S_i)$ for any $1\leq i \leq n$, then the map $A_i\mapsto p \beta(S_i)^*p$ extends to a representation of $P_n$ on $p\scrH.$
\end{proposition}

\begin{proof}
We prove it for $n=2$ but the argument generalises easily.
Consider a representation $\alpha: P_2\to B(\fH)$ defining a Pythagorean pair $(A_1,A_2).$
As usual, denote by $\scrH$ the inductive limit of the $\fH_t$ associated to $(A_1,A_2).$
If $s,t$ are trees, then we write $s+t$ for the tree obtained by attaching $s$ (resp. $t$) to the left (resp. right) leaf of the tree $f_{1,1}$ with two leaves.
If $(s,\xi)\in\fH_s, (t,\eta)\in\fH_t$ we form the vector $(s+t , \xi\oplus \eta)$ in $\fH_{s+t}$.
Define the operators
$$C_1:\scrH\to\scrH, (s+t , \xi\oplus \eta)\mapsto (s,\xi) \text{ and } C_2:\scrH\to\scrH, (s+t , \xi\oplus \eta)\mapsto (t,\eta).$$
They are well defined on the Hilbert space limit $\scrH$ and satisfy 
$$C_1^*(s,\xi) = (s+t,\xi\oplus 0) \text{ and } C_2^*(t,\eta) = (s+t,0\oplus \eta), \ \forall s,t\in\fT, \xi\in\fH_s,\eta\in\fH_t.$$
We obtain that 
$$ C_i C_j^* = \delta_{ i , j } \id \text{ and }  C_1^*C_1 + C_2^*C_2 = \id.$$
Therefore, the map $S_i\mapsto C_i^*$ extends to a representation $\beta:O_2\to B(\scrH).$
Moreover, one can check that $pC_ip = C_ip = A_i$ for $i=1,2.$

The converse is obvious.
\end{proof}

We deduce that any Pythagorean representation can be obtained via a representation of the Cuntz algebra as follows.

Consider a representation $\beta:O_n\to B(\scrH)$ giving a Pythagorean representation $(\pi_\beta,H_\beta)$ with functor $\Phi_\beta.$
The maps $\Phi_\beta(f_i)$ are invertible implying that the direct limit Hilbert space $H_\beta$ is naturally identified with the original Hilbert space $\scrH$ via the following unitary transformation
$$U_\beta:H_\beta\to \scrH, (t,\oplus_\ell\xi_\ell)\in\scrH_t\mapsto \sum_{\ell \text{ a leaf of } t} (\cA^t_\ell)^*\xi_\ell,$$
so there are universal representations $\sigma$ of $F_n$ and $T_n$ inside $O_n$ obtained by using $A_i=S_i^*$ in the formula \ref{pathoperator} and sending $\displaystyle \frac{t}{s} $ to $\sum_{\ell} \mathcal B^{t,s}_\ell\in O_n$.
This  embedding of Thompson's group $T_n$ inside the unitary group of  $O_n$ was discovered by Nekrashevych \cite{Nekrashevych04}.
Hence, for each representation $\beta:O_n\to B(\scrH)$ we have a representation of Thompson's group $T_n$ on $\scrH$ given by $\beta\circ\sigma.$
Note that we have the equality $$\beta\circ\sigma=\Ad(U_\beta)\circ \pi_\beta.$$
Consider a Pythagorean $n$-tuple $(A_1,\cdots, A_n)$ acting on $\fH$ and with representation $(\pi_A,\scrH)$ of $T_n.$
Let $\beta:O_n\to B(\scrH)$ be the representation constructed in the proof of Proposition \ref{prop:Cuntz} satisfying $p\beta(S_i)^*p=p\beta(S_i)^*=A_i$ where $p:\scrH\to \fH$ is the orthogonal projection onto $\fH.$
This provides a functor $\Phi_\beta$ and another unitary representation $(\pi_\beta,H_\beta)$ of $T_n$.
A careful check shows that $$\pi_A = \Ad(U_\beta)\circ\pi_\beta=\beta\circ\sigma.$$

\end{document}